\documentclass[12pt,a4paper,reqno]{amsart}

\usepackage{a4wide}
\usepackage{amssymb,amsmath,amsthm,mathrsfs}
\usepackage{graphicx}
\usepackage{float}
\usepackage{colortbl}
\usepackage{enumerate}
\usepackage{upref}
\usepackage[backref,bookmarks=false]{hyperref}
\usepackage[bookmarks=false]{hyperref}

\theoremstyle{plain}% default
\newtheorem{theorem}{Theorem}
\newtheorem{lemma}{Lemma}[section]

\newtheorem{proposition}{Proposition}
\newtheorem{corollary}[theorem]{Corollary}
\newtheorem*{question}{Question}
\theoremstyle{definition}
\newtheorem{definition}{Definition}%[section]

\newtheorem*{condition}{Condition}

\theoremstyle{remark}
\newtheorem{remark}{Remark}%[section]

\def\R{\ensuremath{\mathbb R}}
\def\N{\ensuremath{\mathbb N}}

\def\I{\ensuremath{{\bf 1}}}
\def\e{{\ensuremath{\rm e}}}

\def\U{\ensuremath{\mathcal U}}
\def\B{\ensuremath{\mathcal B}}
\def\M{\ensuremath{\mathcal M}}

\def\l{{\rm Leb}}
\def\P{\ensuremath{\mathcal P}}
\def\p{\ensuremath{\mathbb P}}

\def\QQ{\ensuremath{\mathscr Q}}

\def\n{\ensuremath{n}}

\def\X{\mathcal{X}}
\def\ie{{\em i.e.}, }

\def\dist{\ensuremath{\text{dist}}}

\def\eps{\varepsilon}

\def\cyl{\text{Z}}
\def\spp{\text{SP\negmedspace}_{p,
\theta}}
\def\sppi{\text{SP\negmedspace}_{\mathbf p_i,
\Theta_i}^{\;\;(i)}}
\def\sppj{\text{SP\negmedspace}_{\mathbf p_j,
\Theta_j}^{\;\;(j)}}
\def\mpp{\text{MP\negmedspace}_{p,
\theta}}
\def\AIM{\text{AIM}}

\numberwithin{equation}{section}
\def\size{\hbar}

\begin{document}

\title{Extremal Index, Hitting time statistics and periodicity}

\author[A. C. M. Freitas]{Ana Cristina Moreira Freitas}
\address{Ana Cristina Moreira Freitas\\ Centro de Matem\'{a}tica \&
Faculdade de Economia da Universidade do Porto\\ Rua Dr. Roberto Frias \\
4200-464 Porto\\ Portugal} \email{amoreira@fep.up.pt}

\author[J. M. Freitas]{Jorge Milhazes Freitas}
\address{Jorge Milhazes Freitas\\ Centro de Matem\'{a}tica da Universidade do Porto\\ Rua do
Campo Alegre 687\\ 4169-007 Porto\\ Portugal}
\email{jmfreita@fc.up.pt}
\urladdr{http://www.fc.up.pt/pessoas/jmfreita}

\author[M. Todd]{Mike Todd}
\address{Mike Todd\\ Mathematical Institute\\
University of St Andrews\\
North Haugh\\
St Andrews\\
KY16 9SS\\
Scotland \\} \email{mjt20@st-andrews.ac.uk }
\urladdr{http://www.mcs.st-and.ac.uk/~miket/}

\thanks{ACMF was partially supported by FCT grant SFRH/BPD/66174/2009. JMF was partially supported by FCT grant SFRH/BPD/66040/2009. MT was partially supported by FCT grant SFRH/BPD/26521/2006 and NSF grants DMS 0606343 and DMS 0908093.  All three authors are supported by FCT (Portugal) projects PTDC/MAT/099493/2008 and PTDC/MAT/120346/2010, which are financed by national and European Community structural funds through the programs  FEDER and COMPETE . All three authors were also supported by CMUP, which is financed by FCT (Portugal) through the programs POCTI and POSI, with national
and European Community structural funds.}

%\date{\today}


%37A50  	Relations with probability theory and stochastic processes
%60G70  	Extreme value theory; extremal processes
%37B20  	Notions of recurrence
%60G10  	Stationary processes
%37C25  	Fixed points, periodic points, fixed-point index theory
%--------
%37A25  	Ergodicity, mixing, rates of mixing
%37D25  	Nonuniformly hyperbolic systems (Lyapunov exponents, Pesin theory, etc.)
%37D35  	Thermodynamic formalism, variational principles, equilibrium states
%37C40  	Smooth ergodic theory, invariant measures

\begin{abstract}
The extremal index appears as a parameter in Extreme Value Laws for stochastic processes, characterising the clustering of extreme events.  We apply this idea in a dynamical systems context to analyse the possible Extreme Value Laws for the stochastic process generated by observations taken along dynamical orbits with respect to various measures.  We derive new, easily checkable, conditions which identify Extreme Value Laws with particular extremal indices.  In the dynamical context we prove that the extremal index is associated with periodic behaviour.   The analogy of these laws in the context of Hitting Time Statistics, as studied in the authors' previous works on this topic, is explained and exploited extensively allowing us to prove, for the first time, the existence of Hitting Time Statistics for balls around periodic points. Moreover, for very well behaved systems (uniformly expanding) we completely 	characterise the extremal behaviour by proving that either we have an Extremal Index less than 1 at periodic points or equal to 1 at any other point.   
This theory then also applies directly to general stochastic processes, adding both useful tools to identify the extremal index and giving deeper insight into the periodic behaviour it suggests.
\end{abstract}

\maketitle

%\tableofcontents

\section{Introduction}

The study of extreme or rare events is of great importance in a wide variety of fields and is often tied in with risk assessment. This explains why Extreme Value Laws (EVL) and the estimation of the tail distribution of the maximum of a large number of observations has drawn much attention and become a highly developed subject.

In many practical situations, such as in the analysis of financial markets or climate phenomena, time series can be modelled by a dynamical system which describes its time evolution. The recurrence effect introduced by Poincar\'e, which is present in chaotic systems, is the starting point for a deeper analysis of the limit distribution of the elapsed time until the occurrence of a rare event, which is usually referred to as Hitting Time Statistics (HTS) and Return Time Statistics (RTS).

In \cite{FFT10}, we established the connection between the existence of EVL and HTS/RTS for stochastic processes arising from discrete time chaotic dynamical systems.  This general link allowed us to obtain results of EVL using tools from HTS/RTS and the other way around (this was applied in cases where the extremal index was 1, which is the most classical setting).

The extremal index (EI) $\theta\in [0,1]$ is a measure of clustering of extreme events, the lower the index, the higher the degree of clustering.
In this paper, we give general conditions to prove the existence of an extremal index $0<\theta<1$,
which can be applied to any stationary stochastic process.   Although our results apply to general stationary stochastic processes, we will be particularly interested in the case where the stochastic process arises from a discrete time dynamical system.  This setup will provide not only a huge diversity of examples, but also a motivation for the conditions we propose, as well as a better understanding of their implications. Namely, motivated by the study of stochastic processes arising from chaotic dynamical systems, we associate the extremal index to the occurrence of periodic phenomena. We will illustrate these results by applying them to time series provided by deterministic dynamical systems as well as to cases where the extremal index is already well understood: an Autoregressive (AR) process introduced by Chernick  and two Maximum Moving Averages (MMA) processes.

Because our conditions on the time series data which guarantee an EVL with a given EI are so general, in the dynamical systems context we are able to prove strong results on EVLs around periodic points.  For example, this allows us to consider non-uniformly hyperbolic dynamical systems.  Moreover, coupling these weak conditions with the connection of EVLs to HTS/RTS enables us to consider hits/returns to balls, rather than cylinders. To our knowledge this is the first result of HTS/RTS different from the standard exponential which applies to balls.  We do this first for so-called `Rychlik systems' which are a very general form of uniformly expanding interval map. As explained in Remark~\ref{rem:multidimensional}, these results can easily be extended to some higher dimensional version of  these Rychlik systems. We also give an example of non-uniformly hyperbolic dynamical system: the full quadratic map (also known as the quadratic Chebyshev polynomial), where invariant measures are absolutely continuous w.r.t. Lebesgue or are, more generally, equilibrium states w.r.t. certain potentials.  In future work we will apply these ideas to even more badly behaved non-uniformly hyperbolic systems. 

One of the striking results here is that, at least for well-behaved systems, an extremal index different from 1 can \emph{only} occur at periodic points.  We prove this for the full shift equipped with the Bernoulli measure, (we believe that this last result holds in greater generality, but do not prove that here). Hence, this result raises the following:
\begin{question}
Is it possible to prove the existence of an EI in $(0,1)$ without some sort of periodicity?
\end{question}
In a more concrete formulation:
\begin{question}
For stationary stochastic processes arising from chaotic dynamical systems, is it possible to prove the existence of an EI in $(0,1)$, either for EVL or HTS/RTS around non-periodic points?
\end{question}

We finish this subsection by emphasising that our conditions on time series data also apply beyond that given by dynamical systems.  Indeed the dynamical systems approach suggests that in very general settings we should view data with an extremal index $\theta \in (0, 1)$ as having some underlying periodic phenomenon.   The conditions we use to check this, which are, to our knowledge, the weakest of their kind, and can almost be reduced to simply checking periodicity and mixing.

Throughout this paper the notation $A(u)\sim B(u)$, for $u$ approaching $u_0$, means that $\lim_{u\to u_0} \frac{A(u)}{B(u)}=1$.  When $u=n$ and $u_0=\infty$ we will just write $A(n)\sim B(n)$. The notation $A(n)=o(n)$ means that $\lim_{n\to \infty} \frac{A(n)} n\to0$.
Also, let $[x]$ denote the integer part of the positive real number $x$ and for a set $A$, the notation $A^c$ will denote the complement of the set $A$.

\subsection{Extreme Value Theory for both independent and dependent stochastic processes}
\label{subsec:definitions}

From here on, the sequence $X_0, X_1, X_2,\ldots$ will always denote a stationary stochastic process with marginal distribution function (d.f.) $F$, \ie $F(x)=\p(X_0\leq x)$. Let $$\bar{F}=1-F$$ and $u_F$ denote the right endpoint of the d.f. $F$, \ie
$
u_F=\sup\{x: F(x)<1\}.
$
We have an \emph{exceedance} of the level $u\in\R$ at time $j\in\N$ if the event $\{X_j>u\}$ occurs.
Define a new sequence of random variables (r.v.)  $M_1, M_2,\ldots$ given by
\begin{equation*}
M_n=\max\{X_0,\ldots,X_{n-1}\}.
\end{equation*}
\begin{definition}
We say that we have an \emph{Extreme Value Law} (EVL) for $M_n$ if there is a non-degenerate d.f. $H:\R\to[0,1]$ with $H(0)=0$ and,  for every $\tau>0$, there exists a sequence of levels $u_n=u_n(\tau)$, $n=1,2,\ldots$,  such that
\begin{equation}
\label{eq:un}
  n\p(X_0>u_n)\to \tau,\;\mbox{ as $n\to\infty$,}
\end{equation}
and for which the following holds:
\begin{equation}
\label{eq:EVL-law}
\p(M_n\leq u_n)\to \bar H(\tau),\;\mbox{ as $n\to\infty$.}
\end{equation}
\end{definition}

In the case $X_0, X_1, X_2,\ldots$ are independent and identically distributed (i.i.d.) r.v. then since $\p(M_n\leq u_n)=(F(u_n))^n$ we have $$\log(\p(M_n\leq u_n))=n\log\left(1-\p(X_0>u_n)\right) \sim -n\p(X_0>u_n),$$ which implies that if \eqref{eq:un} holds, then \eqref{eq:EVL-law} holds with $\bar H(\tau)=\e^{-\tau}$ and vice versa (see \cite[Theorem~1.5.1]{LLR83}).

When $X_0, X_1, X_2,\ldots$ are not independent but satisfy some mixing condition $D(u_n)$ introduced by Leadbetter in \cite{L73} then something can still be said about $H$. Let $F_{i_1,\ldots,i_n}$
denote the joint d.f. of $X_{i_1},\ldots,X_{i_n}$, and set
$F_{i_1,\ldots,i_n}(u)=F_{i_1,\ldots,i_n}(u,\ldots,u)$.
\begin{condition}[$D(u_n)$]\label{cond:D} We say that $D(u_n)$ holds
for the sequence $X_0,X_1,\ldots$ if for any integers
$i_1<\ldots<i_p$ and $j_1<\ldots<j_k$ for which $j_1-i_p>m$, and any
large $n\in\N$,
\[
\left|F_{i_1,\ldots,i_p,j_1,\ldots,j_k}(u_n)-F_{i_1,\ldots,i_p}(u_n)
F_{j_1,\ldots,j_k}(u_n)\right|\leq \gamma(n,m),
\]
where $\gamma(n,m_n)\xrightarrow[n\to\infty]{}0$, for some sequence
$m_n=o(n)$.
\end{condition}
If $D(u_n)$ holds for $X_0,X_1,\ldots$ and the limit \eqref{eq:EVL-law} exists for some $\tau>0$ then there exists $0\leq\theta\leq1$ such that $\bar H(\tau)=\e^{-\theta\tau}$ for all $\tau>0$ (see \cite[Theorem~2.2]{L83} or \cite[Theorem~3.7.1]{LLR83}).
\begin{definition}
We say that $X_0,X_1,\ldots$ has an \emph{Extremal Index} (EI) $0\leq\theta\leq1$ if we have an EVL for $M_n$ with $\bar H(\tau)=\e^{-\theta\tau}$ for all $\tau>0$.
\end{definition}

The notion of the EI was latent in the work of Loynes \cite{L65} but was established formally by Leadbetter in \cite{L83}. It gives a measure of the strength of the dependence of $X_0,X_1,\ldots$, so that $\theta=1$ indicates that the process has practically no memory while $\theta=0$, conversely, reveals extremely long memory.  Another way of looking at the EI is that it gives some indication on how much exceedances of high levels have a tendency to ``cluster''. Namely, for $\theta>0$ this interpretation of the EI is that $\theta^{-1}$ is the mean number of exceedances of a high level in  a cluster of large observations, \ie is the ``mean size of the clusters''.

\begin{remark}
\label{rem:sequences-un}
The sequences of real numbers $u_n=u_n(\tau)$, $n=1,2,\ldots$, are usually taken to be one parameter linear families such as $u_n=a_n y +b_n$, where $y\in\R$ and $a_n>0$, for all $n\in\N$. Observe that $\tau$ depends on $y$ through $u_n$ and, in fact, in the i.i.d. case, depending on
the tail of the marginal d.f. $F$, we have that $\tau=\tau(y)$ is
of one of the following three types (for some $\alpha>0$):
\begin{equation*}\tau_1(y)=\e^{-y} \text{ for } y \in
\mathbb{R},\quad \tau_2(y)=y^{-\alpha} \text{ for } y>0\quad \text{ and }\quad
\tau_3(y)=(-y)^{\alpha} \text{ for } y\leq0.
\end{equation*}
\end{remark}

\subsection{Hitting and return time statistics}

Consider a deterministic discrete time dynamical system $(\X,\B,\mu,f)$, where $\X$ is topological space, $\mathcal B$ is the
Borel $\sigma$-algebra, $f:\X\to\X$ is a measurable map and $\mu$ is an $f$-invariant probability measure, \ie $\mu(f^{-1}(B))=\mu(B)$, for all $B\in\B$. One can think of $f:\X\to\X$ as the evolution law that establishes how time affects the transitions from one state in $\X$ to another.

Consider now a set $A\in\B$ and a new r.v. that we refer to as \emph{first hitting time} to $A$ and denote by $r_A:\X\to\N\cup\{+\infty\}$ where
\begin{equation*}
r_A(x)=\min\left\{j\in\N\cup\{+\infty\}:\; f^j(x)\in A\right\}.
\end{equation*}
Given a sequence of sets $\{U_n\}_{n\in \N}$ so that
$\mu(U_n) \to 0$ we consider the sequence of r.v. $r_{U_1}, r_{U_2},\ldots$ If under suitable normalisation $r_{U_n}$ converges in distribution to some non-degenerate d.f. $G$ we say that the system has \emph{Hitting Time Statistics} (HTS) $G$ for $\{U_n\}_{n\in\N}$. For systems with `good mixing properties', $G$ is the standard exponential d.f., in which case, we say that we have \emph{exponential HTS}.

We say that the system has HTS $G$ to balls at $\zeta$ if for any sequence $(\delta_n)_{n\in \N}\subset \R^+$ such that $\delta_n\to 0$ as $n\to \infty$ we have HTS $G$ for $(U_n)_n=(B_{\delta_n}(\zeta))_n$.

Let $\P_0$ denote a partition of $\X$. We define the corresponding pullback partition $\P_n=\bigvee_{i=0}^{n-1} f^{-i}(\P_0)$, where $\vee$ denotes the join of partitions. We refer to the elements of the partition $\P_n$ as \emph{cylinders of order $n$}. For every $\zeta\in\X$, we denote by $\cyl_n[\zeta]$ the cylinder of order $n$ that contains $\zeta$. For some $\zeta\in \X$ this cylinder may not be unique, but we can make an arbitrary choice, so that $\cyl_n[\zeta]$ is well defined. We say that the system has HTS $G$ to cylinders at $\zeta$ if  we have HTS $G$ for $U_n=\cyl_n(\zeta)$.

Let $\mu_A$ denote the conditional measure on $A\in\B$, \ie $\mu_A:=\frac{\mu|_A}{\mu(A)}$. Instead of starting  somewhere in the whole space $\X$, we may want to start in $U_n$ and study the fluctuations of the normalised return time to $U_n$ as $n$ goes to infinity, \ie for each $n$, we look at the random variables $r_{U_n}$ as being defined in the probability space $(U_n,\B\cap U_n, \mu_{U_n})$ and wonder if, under some normalisation, they converge in distribution to some non-degenerate d.f. $\tilde G$, in which case, we say that the system has \emph{Return Time Statistics} (RTS) $\tilde G$ for $\{U_n\}_{n\in\N}$. The existence of exponential HTS is equivalent to the existence of exponential RTS. In fact, according to the Main Theorem in \cite{HLV05}, a system has HTS $G$ if and only if it has RTS $\tilde G$ and
\begin{equation}
\label{eq:HTS-RTS}
G(t)=\int_0^t(1-\tilde G(s))\,ds.
\end{equation}

Regarding normalising sequences to obtain HTS/RTS, we recall Kac's Lemma, which states that the expected value of $r_A$ with respect to $\mu_A$ is  $\int_A r_A~d\mu_A =1/\mu(A)$.  So in studying the fluctuations of $r_A$ on $A$, the relevant normalising factor should be $1/\mu(A)$.
\begin{definition}
\label{def:HTS/RTS}
Given a sequence of sets $(U_n)_{n\in \N}$ so that
$\mu(U_n) \to 0$, the system has \emph{HTS}
$G$ for $(U_n)_{n\in \N}$ if for all $t\ge 0$
\begin{equation}\label{eq:def-HTS-law}
\mu\left(r_{U_n}\leq\frac t{\mu(U_n)}\right)\to G(t) \;\mbox{ as $n\to\infty$,}
\end{equation}
and the system has \emph{RTS}
$\tilde G$ for $(U_n)_{n\in \N}$ if for all $t\ge 0$
\begin{equation}\label{eq:def-RTS-law}
\mu_{U_n}\left(r_{U_n}\leq\frac t{\mu(U_n)}\right)\to\tilde G (t)\;\mbox{ as $n\to\infty$}.
\end{equation}
\end{definition}

The theory of HTS/RTS laws is now a well developed theory, applied first to cylinders and hyperbolic dynamics, and then extended to balls and also to non-uniformly hyperbolic systems. We refer to \cite{C00} and \cite{S09} for very nice reviews as well as many references on the subject.  (See also \cite{AG01}, where the focus is more towards a finer analysis of uniformly hyperbolic systems.)  Since the early papers \cite{P91,H93}, several different approaches have been used to prove HTS/RTS: from the analysis of adapted Perron-Frobenius operators as in \cite{H93}, the use of inducing schemes as in \cite{BSTV03}, to the relation between recurrence rates and dimension as explained in \cite[Section~4]{S09}.

For many mixing systems it is known that the HTS/RTS are standard exponential around almost every point. Among these systems we note the following:
Markov chains \cite{P91}, Axiom A diffeomorphisms \cite{H93}, uniformly expanding maps of the interval \cite{C96}, 1-dimensional non-uniformly expanding maps \cite{HSV99,BSTV03,BV03,BT09}, partially hyperbolic dynamical systems \cite{D04}, toral automorphisms \cite{DGS04},  higher dimensional non-uniformly hyperbolic systems (including H\'enon maps) \cite{CC12}.

However, even for systems with good mixing properties, it is known at least since \cite{H93} that at some special (periodic) points,  similar distributions for the HTS/RTS (for cylinders) with an exponential parameter $0<\theta<1$ (\ie $1-G(t)=\e^{-\theta t}$) apply. This subject was studied, also in the cylinder context, in \cite{HV09}, where the sequence of successive returns to neighbourhoods of these points was proved to converge to a compound Poisson process.

\subsection{The connection between EVL and HTS/RTS}
We start by explaining what we mean by stochastic processes arising from discrete time dynamical systems. Take a system $(\X,\B,\mu,f)$ and consider the time series $X_0, X_1, X_2,\ldots$ arising from such a system simply by evaluating a given random variable $\varphi:\X\to\R\cup\{\pm\infty\}$ along the orbits of the system, or in other words, the time evolution given by successive iterations by $f$:
\begin{equation}
\label{eq:def-stat-stoch-proc-DS} X_n=\varphi\circ f^n,\quad \mbox{for
each } n\in {\mathbb N}.
\end{equation}
Clearly, $X_0, X_1,\ldots$ defined in this way is not an independent sequence.  However, $f$-invariance of $\mu$ guarantees that this stochastic process is stationary. We assume that $\varphi$ achieves a global maximum at $\zeta\in\X$ and the event $\{x\in\X:\; \varphi(x)>u\}=\{X_0>u\}$ corresponds to a topological ball ``centred'' at $\zeta$. EVLs for the partial maximum of such sequences have been proved directly in the recent papers \cite{C01,FF08,FFT10,HNT12,GHN11,FFT11}. We highlight the pioneer work of Collet \cite{C01} for the innovative ideas introduced. The dynamical systems covered in these papers include non-uniformly hyperbolic 1-dimensional maps (in all of them), higher dimensional non-uniformly expanding maps in \cite{FFT10}, suspension flows in \cite{HNT12}, billiards and Lozi maps in \cite{GHN11}.

In \cite{FFT10}, we formally established the link between EVL and HTS/RTS (for balls) of stochastic processes given by \eqref{eq:def-stat-stoch-proc-DS}. Essentially, we proved that if such time series have an EVL $H$ then the system has HTS $H$ for balls ``centred'' at $\zeta$ and vice versa. Recall that having HTS $H$ is equivalent to say that the system has RTS $\tilde H$, where $H$ and $\tilde H$ are related by \eqref{eq:HTS-RTS}. This was based on the elementary observation that for stochastic processes given by \eqref{eq:def-stat-stoch-proc-DS} we have:
\begin{equation}
\label{eq:rel-Mn-r}
\{M_n\leq u\}=\{r_{\{X_0>u\}}>n\}.
\end{equation}
We exploited this connection to prove EVL using tools from HTS/RTS and the other way around. In  \cite{FFT11}, we carried the connection further to include the cases where the invariant measure $\mu$ may not be absolutely continuous with respect to Lebesgue measure and also to understand HTS/RTS for cylinders rather than balls in terms of EVL. To achieve the latter we introduced the notion of \emph{cylinder EVL} which essentially requires that the limits \eqref{eq:un} and \eqref{eq:EVL-law} exist only for particular time subsequences $\{\omega_j\}_{j\in\N}$ of $\{n\}_{n\in\N}$ (see Section~\ref{sec:cyl}).

Hence, under the conditions of \cite[Theorem~2]{FFT10}, when $X_0, X_1, X_2,\ldots$ has an EI $\theta<1$ then we have HTS for balls $G$ given by
\begin{equation}
\label{eq:HTS-EI}
G(\tau)=1-\e^{-\theta \tau}.
\end{equation}
Using \eqref{eq:rel-Mn-r} plus the integral relation \eqref{eq:HTS-RTS} and arguing as in the proof of \cite[Theorem~2]{FFT10}, we have RTS for balls $\tilde G$ that can be written as:
\begin{equation}
\label{eq:RTS-EI}
\tilde G(\tau)=\lim_{n\to\infty}\mu(M_n>u_n | X_0>u_n )= (1-\theta)+\theta(1-\e^{-\theta \tau}),
\end{equation}
or in other words: the return time law is the convex combination of a Dirac law at zero and an exponential law of average $\theta^{-1}$ where the weight is the EI $\theta$ itself.

As a consequence of this relation new light can be brought to the work of Galves and Schmitt \cite{GS90} who introduced a short correction factor $\lambda$ in order to get exponential HTS, that was then studied later in great detail by Abadi \emph{et al.} \cite{AG01, A04, A06, AV08, AVe09, AS11}, and which, in case of being convergent, can be seen as the EI. This means that we now have two different perspectives from which to look at the aforementioned papers, as well as to look at the work of those who developed the probabilistic theory of the EI such as \cite{L83, LLR83, O87, LR88, HHL88, LN89, CHM91}. Just to give an example of the advantage of realising this connection, we observe that O'Brien's formula for the EI in \cite{O87}, which is widely used in the estimation of the EI, can be easily derived from formula \eqref{eq:RTS-EI} for the RTS.

\subsection{Extreme Value Laws in the absence of clustering}
\label{subsec:no-clustering}
In this subsection we recall some of the results which imply the existence of EVLs in the \emph{absence} of clustering, which means the EI is 1. We do so to motivate and provide a better understanding of the conditions we propose in Section~\ref{sec:EI-Periodicity}. 

We start by recalling a condition proposed by Leadbetter for general stochastic processes which imposes some sort of independence on the short range that prevents the appearance of clustering.
Supposing that $D(u_n)$ holds, let $(k_n)_{n\in\N}$ be a sequence of integers such that
\begin{equation}
\label{eq:kn-sequence-1}
k_n\to\infty\quad \mbox{and}\quad  k_n t_n = o(n).
\end{equation}
\begin{condition}[$D'(u_n)$]\label{cond:D'} We say that $D'(u_n)$
holds for the sequence $X_0,X_1,X_2,\ldots$ if there exists a sequence $\{k_n\}_{n\in\N}$ satisfying \eqref{eq:kn-sequence-1} and such that
\begin{equation}
\label{eq:D'un}
\lim_{n\rightarrow\infty}\,n\sum_{j=1}^{[n/k_n]}\p( X_0>u_n,X_j>u_n)=0.
\end{equation}
\end{condition}
According to \cite[Theorem~1.2]{L83}, if conditions $D(u_n)$ and $D'(u_n)$ hold for $X_0, X_1,\ldots$ then there exists an EVL for $M_n$ and $H(\tau)=1-e^{-\tau}$.

However, when one considers stochastic processes arising from dynamical systems such as in  \eqref{eq:def-stat-stoch-proc-DS}, in practice condition $D(u_n)$ can not be verified unless the system satisfies some strong uniformly mixing condition such as $\alpha$-mixing (see \cite{B05} for a definition), and even in these cases it can only be verified for certain subsequences of $\{n\}_{n\in\N}$, which means that the limit laws only hold for cylinders.  For that reason, based on the work of Collet \cite{C01}, in \cite{FF08a}  we proposed a condition we called $D_2(u_n)$ which is much weaker than $D(u_n)$, and which follows from sufficiently fast decay of correlations, thus allowing us to obtain the results for balls rather than cylinders.

We remark that rates of decay of correlations are nowadays very well known for a wide variety of systems including non-uniformly hyperbolic systems admitting a Young tower (see \cite{Y98,Y99}).

\begin{condition}[$D_2(u_n)$]\label{cond:D2} We say that $D_2(u_n)$ holds for the sequence $X_0,X_1,\ldots$ if for all $\ell,t$
and $n$
\begin{align*}
|\p\left\{X_0>u_n\cap
  \max\{X_{t},\ldots,X_{t+\ell-1}\leq u_n\}\right\}-\p\{X_0>u_n\}
  \p\{M_{\ell}\leq u_n\}|\leq \gamma(n,t),
\end{align*}
where $\gamma(n,t)$ is decreasing in $t$ for each $n$ and
$n\gamma(n,t_n)\to0$ when $n\rightarrow\infty$ for some sequence
$t_n=o(n)$.
\end{condition}

Observe that while $D(u_n)$ imposes some rate for the independence of two blocks of r.v. separated by a time gap which is independent of the size of the blocks, condition $D_2(u_n)$ requires something similar but only when the first block is reduced to one r.v. only. This detail turns out to be crucial when proving $D_2(u_n)$ from decay of correlations as can be seen in \cite[Section~2]{FF08a}. The interesting fact is that we can replace $D(u_n)$ by $D_2(u_n)$ in \cite[Theorem~1.2]{L83} and the conclusion still holds. In fact, according to \cite[Theorem~1]{FF08a}, if conditions $D_2(u_n)$ and $D'(u_n)$ hold for $X_0, X_1,\ldots$ then there exists an EVL for $M_n$ and $H(\tau)=1-e^{-\tau}$. The idea is that condition $D'(u_n)$, instead of being used once as in the original proof of Leadbetter, is used twice: in one of the instances it is used in conjunction with $D_2(u_n)$ to produce the same effect as $D(u_n)$ alone.

Basically this means that as long as you start with a dynamical system with sufficiently fast decay of correlations you only have to prove $D'(u_n)$ to show the existence of exponential EVL or HTS/RTS.   

\subsection{Structure of the paper}

The paper is organised as follows. In Section~\ref{sec:EI-Periodicity} we give conditions to prove the existence of an EI for general stochastic processes; initially this is applied to `first order' clustering behaviour and then later to higher order clustering. In Section~\ref{sec:EI-DS} we give a very general introduction to the dynamical systems and accompanying measures we will be studying.  We also explain the link between EVL and HTS and state general theorems for those laws in this context.  In Section~\ref{sec:dyn egs} we give some concrete examples of dynamical systems, measures and observables yielding EVLs with EI in $(0,1)$.  These examples are so-called Rychlik systems as well as the full quadratic map.  Section~\ref{sec:cyl} is a short section explaining the relevant conditions required to guarantee an EVL for returns to cylinders rather than balls, while Section~\ref{sec:ex ind imples per} shows that in that context we can completely characterise all possible EVLs for simple dynamical systems.  Finally in the appendix we show how our conditions apply to various standard types of random variables not necessarily produced by a dynamical system, namely, two MMA processes and one AR(1) introduced by Chernick in \cite{C81}.

\section{Extremal index and periodicity}
\label{sec:EI-Periodicity}

In this section we give conditions that can be applied to any stationary stochastic process and which allow us to prove the existence of an EI by realising the presence of one or more underlying periodic phenomenon. To explain what is happening here, and to underline the motivation, we first turn to the main stream of the paper which is the dynamics around repelling periodic points. Our strategy is essentially to replace the role of ``exceedances'' (that correspond to entrances in balls) by what we shall call ``escapes'' (that correspond to entrances in annuli), and then reduce to the usual strategy when no clustering occurs,  described in Section~\ref{subsec:no-clustering}.  

\subsection{Motivation from periodic dynamics}
\label{subsec:motivation}

We consider a model case: the stochastic processes defined by \eqref{eq:def-stat-stoch-proc-DS} when $\varphi$ achieves a global maximum at a repelling periodic point $\zeta\in \X$, of prime period $p\in\N$, which is also a Lebesgue density point of an invariant measure $\mu$, where $\mu$ is assumed to be absolutely continuous with respect to Lebesgue. We postpone the exact meaning of all this to Section~\ref{sec:EI-DS} but keep the following facts: 
\begin{enumerate}
\item we assume that for $u$ sufficiently large, $\{X_0>u\}$ corresponds to a topological ball centred at $\zeta$;

\item the periodicity of $\zeta$ implies that for all large $u$, $\{X_0>u\}\cap f^{-p}(\{X_0>u\})\neq\emptyset$ and the fact that the prime period is $p$ implies that $\{X_0>u\}\cap f^{-j}(\{X_0>u\})=\emptyset$ for all $j=1,\ldots,p-1$. 

\item \label{item:repeller} the fact that $\zeta$ is repelling means that we have backward contraction implying that $\bigcap_{j=0}^i f^{-j}(X_0>u)$ is another ball of smaller radius around $\zeta$ and $\l(\bigcap_{j=0}^i f^{-j}(X_0>u))\sim(1-\theta)^i\l(X_0>u)$, for all $u$ sufficiently large and some $0<\theta<1$;

\item \label{item:density-point} the fact that $\zeta$ is a Lebesgue density point of $\mu$ implies that we can replace $\l$ by $\mu$ in the previous item.
 
\end{enumerate}

Note that $Q(u)=\{X_0>u,X_p\leq u\}=\{X_0>u\}\setminus f^{-p}(\{X_0>u\})$ can be seen as an annulus centred at $\zeta$ that corresponds to the points that after $p$ steps manage to escape from $\{X_0>u\}$. Moreover, for $u$ large we have $\mu(Q(u))\sim\theta\mu(X_0>u)$.

Following the work of Hirata \cite{H93} on Axiom A diffeomorphisms, it is known that around periodic points there is a parameter less than 1 in the Hitting Times distribution, which in light of the connection between EVL and HTS can be seen as the Extremal Index. However, this has only been checked for cylinders. The approach we propose here allows us to finally establish the result for balls, and for non-Axiom A systems. 

The main obstacle when dealing with periodic points is that they create plenty of dependence in the short range. In particular, using properties \eqref{item:repeller} and \eqref{item:density-point} we have that for all $u$ sufficiently large
$$
\mu(\{X_0>u\}\cap \{X_p>u\})\sim(1-\theta)\mu(X_0>u).
$$
which implies that $D'(u_n)$ is not satisfied, since for the levels $u_n$ as in \eqref{eq:un} it follows that
$$
n\sum_{j=1}^{[n/k_n]}\mu(X_0>u_n, X_j>u_n)\geq n\mu(X_0>u_n, X_p>u_n)\xrightarrow[n\to\infty]{}(1-\theta)\tau.
$$
Recalling the discussion at the end of Section~\ref{subsec:no-clustering}, condition $D'(u_n)$ was essential to allow the replacement of $D(u_n)$ by $D_2(u_n)$ in order to use decay of correlations to get the result. To overcome this difficulty around periodic points we make a key observation that roughly speaking tells us that around periodic points one just needs to replace the ball $\{X_0>u_n\}$ by the annulus $Q(u_n)$: then much of the analysis works out as in the absence of clustering.  

To be more precise, let $\mathscr Q_n(u_n):=\bigcap_{j=0}^{n-1} f^{-j}(Q(u_n)^c)$. Note that while the occurrence of the event $\{M_n\leq u_n\}$ means that no entrance in the ball $\{X_0>u_n\}$ has occurred up to time $n$, the occurrence of $\mathscr Q_n(u_n)$ means that no entrance in the annulus $Q(u_n)$ has occurred up to time $n$. 
\begin{proposition}
\label{prop:ball-annulus} Let $X_0, X_1,,\ldots$ be a stochastic process defined by \eqref{eq:def-stat-stoch-proc-DS} where $\varphi$ achieves a global maximum at a repelling periodic point $\zeta\in \X$, of prime period $p\in\N$, so that conditions (1) to (4) above hold. Let $(u_n)_n$ be a sequence of levels such that \eqref{eq:un} holds. Then,
$$
\lim_{n\to\infty}\mu(M_n\leq u_n)=\lim_{n\to\infty}\mu(\mathscr Q_n(u_n)).
$$   
\end{proposition}
\begin{proof}
Clearly $$\{M_n\leq u_n\}\subset\QQ_{n}(u_n).$$ Next, note that if $\QQ_{n}(u_n)\setminus \{M_n\leq u_n\}$ occurs, then you must enter the ball $\{X_0>u_n\}$ at some point which means we may define first time it happens by $i=\inf\{j\in\{0,1,\ldots n-1\}:\; X_j>u_n\}$ and let $s_i=[ \frac {n-1-i}{p}]$. However, since $\QQ_{p,0,n}(u_n)$ does occur, you must never enter the annulus $Q_(u_n)$ which is the only way out of the ball $\{X_0>u_n\}$. Hence, once you enter the ball you must never leave it, which means that $f^{-i}\left(\cap_{j=1}^{s_i}f^{-jp}(X_0>u_n)\right)$ must occur. 
 Consequently,
\[
\QQ_{p,0,n}(u_n)\setminus \{M_n\leq u_n\}\subset \bigcup_{i=0}^{n-1}f^{-i}\left(\cap_{j=0}^{s_i}f^{-jp}(X_0>u_n)\right).
\]
It follows by stationarity, properties \eqref{item:repeller}, \eqref{item:density-point} above and \eqref{eq:un} that
\begin{align*}
  \mu(\QQ_{p,0,n}(u_n))-\mu(\{M_n\leq u_n\})&\leq \sum_{i=0}^{n-1}\mu\left( f^{-i}\left(\cap_{j=0}^{s_i}f^{-jp}(X_0>u_n)\right)\right)\\
    &\leq p\sum_{\kappa=0}^{[n/p]}\mu\left(\cap_{j=0}^{\kappa}f^{-jp}(X_0>u_n)\right)\\
    &\lesssim p\sum_{\kappa=0}^{\infty}(1-\theta)^\kappa\mu\left(X_0>u_n\right)\xrightarrow[n\to\infty]{}0.
  \end{align*}
\end{proof}
The proposition above is essentially saying that if the sequence of levels is well chosen then, around repelling periodic points, in the limit, the probability of there being no entrances in the ball $\{X_0>u_n\}$ equals the probability of there being no entrances in the annulus $Q(u_n)$. Then the idea to cope with clustering caused by periodic points is to adapt conditions $D_2(u_n)$ and $D'(u_n)$, letting annuli replace balls. In order to make the theory as general as possible, motivated by the above considerations for stochastic processes generated by dynamical systems around periodic points, we will propose some   abstract conditions to prove the existence of an EI less than 1 for general stationary stochastic processes.  

\subsection{Existence of an EI due to the presence of periodic phenomena}
\label{ssec:spp etc}
 
We start by an abstract condition designed to capture the essential properties (1)-(4) from Section~\ref{subsec:motivation} in order to guarantee that the conclusion of Proposition~\ref{prop:ball-annulus} holds for general stochastic processes.  
It imposes some type of periodic behaviour of period $p\in\N$ plus a summability requirement. For that reason we shall denote it by $\spp$ which stands for Summable Periodicity of period $p$.   To state the condition we will use a sequence of levels $(u_n)_n$ as in \eqref{eq:un}.

\begin{condition}[$\spp(u_n)$]\label{cond:SP} We say that  $X_0,X_1,X_2,\ldots$ satisfies condition $\spp(u_n)$ for $p\in \N$ and $\theta\in [0,1]$ if
\begin{equation}
\label{cond:periodicity}
\lim_{n\to \infty}\sup_{1\le j<p}\p(X_j>u_n| X_0>u_n)=0 \quad \mbox{ and}\quad \lim_{n\to \infty} \p(X_p>u_n| X_0>u_n)\to (1-\theta)
\end{equation}
and moreover
\begin{equation}
\label{cond:summability}
\lim_{n\to \infty}\sum_{i=0}^{[\frac{n-1}p]} \p(X_0>u_n, X_p>u_n, X_{2p}>u_n,\ldots,X_{ip}>u_n)=0.
\end{equation}
\end{condition}
Condition \eqref{cond:periodicity}, when $\theta<1$, imposes some sort of periodicity of period $p$ among the exceedances of high levels $u_n$, since if at some point the process exceeds the high level $u_n$, then, regardless of how high $u_n$ is, there is always a strictly positive probability of another exceedance occurring at the (finite) time $p$.  In fact, if the process is generated by a deterministic dynamical system $f:\X\to \X$ as in    \eqref{eq:def-stat-stoch-proc-DS} and $f$ is continuous then \eqref{cond:periodicity} implies that $\zeta$ is a periodic point of period $p$, \ie $f^p(\zeta)=\zeta$.

We also state a stronger condition, which is often simpler to check than $\spp(u_n)$ and which requires, besides the periodicity, some type of Markov behaviour that which immediately implies the summability condition \eqref{cond:summability}. We call it $\mpp$ which stands for Markovian Periodicity.  We will check this condition rather than $\spp(u_n)$ in the applications presented in Sections~\ref{sec:EI-DS} and \ref{sec:dyn egs} as well as in Appendix~\ref{sec:AR1}.

\begin{condition}[$\mpp(u_n)$]\label{cond:PM}We say that  $X_0,X_1,X_2,\ldots$ satisfies the condition $\mpp(u_n)$ for $p\in \N$ and $\theta\in [0,1]$ if
\begin{gather}
\label{cond:Markovian-Periodicity}
\begin{split}
\lim_{n\to \infty}\sup_{1\le j<p} & \p(X_j>u_n| X_0>u_n)=0  \mbox{ and }\\
& \lim_{n\to \infty} \sup_i\frac{\p(X_p>u_n, X_{2p}>u_n,\ldots,X_{ip}>u_n | X_0>u_n)}{(1-\theta)^i} =1.\end{split}
\end{gather}
\end{condition}
Note that if besides condition \eqref{cond:periodicity}, the stationary stochastic process satisfies the following Markovian property:
\begin{equation}
\label{cond:Markov}
\p(X_{ip}>u | X_{(i-1)p}>u, \ldots, X_0>u)=\p(X_{ip}>u | X_{(i-1)p}>u), \quad \mbox{for all $i\in\N$},
\end{equation}
then it can easily be seen by an induction argument that condition $\mpp(u_n)$ holds.

Assuming that $\spp(u_n)$ holds, for $i,s,\ell\in\N\cup\{0\}$, we define the events:
\begin{equation*}
Q_{p,i}(u):=\{X_i>u, X_{i+p}\leq u\},\; Q_{p,i}^*(u):=\{X_i>u\}\setminus Q_{p,i}(u) \;\mbox{and } \QQ_{p,s,\ell}(u)=\bigcap_{i=s}^{s+\ell-1} Q_{p,i}^c(u).
\end{equation*}
Assuming $\theta<1$, by \eqref{cond:periodicity}, we know that the stochastic process has some underlying periodic behaviour such that the occurrence of an exceedance of a high level $u_n$ at time $i$ leads to another exceedance at time $i+p$, with probability approximately $(1-\theta)$.  Therefore,
\begin{list}{$\bullet$}
{ \itemsep 1.0mm \topsep 0.0mm \leftmargin=7mm}
\item $Q_{p,i}^*(u_n)$ corresponds exactly to the realisations of the process with an exceedance of $u_n$, at time $i$, which were ``\emph{captured}'' by the underlying periodic phenomena; and
\item $Q_{p,i}(u_n)$ corresponds to those realisations with an exceedance of $u_n$, at time $i$, but that manage to ``\emph{escape}'' the periodic behaviour.
\end{list}
Hence, if $Q_{p,i}^*(u_n)$ occurs, then we say we have a \emph{capture} at time $i$ while, if $Q_{p,i}(u_n)$ occurs, then we say we have an \emph{escape} at time $i$.
The event $\QQ_{p,s,\ell}(u_n)$ corresponds to the realisations for which no escapes occur between times $s$ and $s+\ell-1$. Recall that in the terminology used in Subsection~\ref{subsec:motivation} where the occurrence of exceedances correspond to entrances in balls, the occurrence of escapes correspond to entrances in annuli. Note that for either a capture or an escape to occur at time $i$, an exceedance must occur at that time.

Note that if condition $\spp(u_n)$ holds we must have:
\[
n\p(Q_{p,0}^*(u_n))=n\p(X_0>u_n, X_p>u_n)=n\p(X_0>u_n)\p(X_p>u_n | X_0>u_n)\xrightarrow[n\to\infty]{} \tau (1-\theta)
\]
and consequently conclude that under $\spp(u_n)$, we have
\begin{equation}
\label{eq:level-un-annulus}
n\p(Q_{p,0}(u_n))\to \theta \tau,\quad\mbox{as $n\to \infty$.}
\end{equation}

As we will show in Theorem~\ref{thm:existence-EI}, under $\spp(u_n)$ the conclusion of Proposition~\ref{prop:ball-annulus} still holds. This means that, in loose terms, the limit distribution of the exceedances is the same as that of the escapes. Hence, in order to prove the existence of limiting law for the maximum in the presence of a periodic phenomenon creating clustering, we follow a similar strategy to that used in \cite{FF08a}, with escapes playing the role of exceedances. We define:

\begin{condition}[$D^p(u_n)$]\label{cond:Dp}We say that $D^p(u_n)$ holds
for the sequence $X_0,X_1,X_2,\ldots$ if for any integers $\ell,t$
and $n$
\[ \left|\p\left(Q_{p,0}(u_n)\cap
  \QQ_{p,t,\ell}(u_n)\right)-\p(Q_{p,0}(u_n))
  \p(\QQ_{p,0,\ell}(u_n))\right|\leq \gamma(n,t),
\]
where $\gamma(n,t)$ is nonincreasing in $t$ for each $n$ and
$n\gamma(n,t_n)\to0$ as $n\rightarrow\infty$ for some sequence
$t_n=o(n)$.
\end{condition}
This condition requires some sort of mixing by demanding that an escape at time $0$ is an event which gets more and more independent from an event corresponding to no escapes during some period, as the time gap between these two events gets larger and larger.  It is in this condition that the main advantage of our approach to prove the EI lies. This is because in all the approaches we are aware of (see for example \cite{L83, O87, HHL88, LN89, CHM91}), some condition like $D(u_n)$ from Leadbetter \cite{L73} is used. Some are slightly weaker like $\AIM(u_n)$ from \cite{O87} or $\Delta(u_n)$ in \cite{LR98}, but they all have a uniform bound on the ``independence'' of two events separated by a time gap, where both these events may depend on an arbitrarily large number of r.v.s of the sequence $X_0, X_1,\ldots$. In contrast, in our condition $D^p(u_n)$, the first event $Q_{p,0}(u_n)$ depends only on the r.v.s $X_0$ and $X_p$ and this proves to be crucial when applying it to stochastic processes arising from dynamical systems as explained in Subsection~\ref{subsubsec:roles-Dp-D'p}.

Assuming $D^p(u_n)$ holds let $(k_n)_{n\in\N}$ be a sequence of integers such that
\begin{equation}
\label{eq:kn-sequence}
k_n\to\infty\quad \mbox{and}\quad  k_n t_n = o(n). 
\end{equation}

\begin{condition}[$D'_p(u_n)$]\label{cond:D'p} We say that $D'_p(u_n)$
holds for the sequence $X_0,X_1,X_2,\ldots$ if there exists a sequence $\{k_n\}_{n\in\N}$ satisfying \eqref{eq:kn-sequence} and such that
\begin{equation}
\label{eq:D'rho-un}
\lim_{n\rightarrow\infty}\,n\sum_{j=1}^{[n/k_n]}\p( Q_{p,0}(u_n)\cap
Q_{p,j}(u_n))=0.
\end{equation}
\end{condition}
This last condition is very similar to Leadbetter's $D'(u_n)$ from \cite{L83}, except that instead of preventing the clustering of exceedances it prevents the clustering of escapes by requiring that they should appear scattered fairly evenly through the time interval from $0$ to $n-1$.

Our main result in this section is the following:
\begin{theorem}
  \label{thm:existence-EI}
  Let $(u_n)_{n\in\N}$ be such that $n\p(X>u_n)=n(1-F(u_n))\to\tau$,
  as $n\to\infty$, for some $\tau\geq 0$.
  Consider a stationary stochastic process $X_0, X_1, X_2,\ldots$ satisfying $\spp(u_n)$ for some $p\in\N$, and $\theta\in (0,1)$.
  Assume further that conditions
  $D^p(u_n)$ and $D'_p(u_n)$ hold. Then
  \begin{equation}
  \label{eq:max-ei}
   \lim_{n\to\infty}\p(M_n\leq u_n)=\lim_{n\to\infty}\p(\QQ_{p,0,n}(u_n))=\e^{-\theta \tau}.
  \end{equation}
\end{theorem}

Theorem~\ref{thm:existence-EI} and in particular formula \eqref{eq:max-ei} allow us to paint the following picture:  In \eqref{eq:RTS-EI} we are concerned with the distribution of $M_n$ given that an exceedance of level $u_n$ has occurred at time 0.  The underlying periodic phenomena and in particular the capture incidents are responsible for the appearance of the Dirac term in \eqref{eq:RTS-EI} for the distribution of RTS with a weight given by the probability of a capture occurring, given that an exceedance has occurred, which is $1-\theta$. On the other hand, the escapes are responsible for the appearance of the exponential term in \eqref{eq:RTS-EI}, again with a weight given by the probability of an escape occurring, given that an exceedance has occurred, which is $\theta$. However, the distribution of $M_n$, where we assume nothing about exceedances at time 0, is equal to the one of the HTS, which, as can be seen in \eqref{eq:HTS-EI}, only sees the exponential term or in other words the escape component. Formula \eqref{eq:max-ei} is then saying that computing the distribution of $M_n$ can be reduced to computing the distribution of the escapes.
\begin{remark}
If we enrich the process and the statistics by
considering either multiple returns or multiple exceedances we can study Exceedance Point Processes or Hitting Times Point Processes as in  \cite[Section~3]{FFT10}. One would expect these point processes to converge to a compound Poisson process, consisting, in loose terms, of a limiting Poisson process ruling the cluster positions, to which is associated a multiplicity corresponding to the cluster size. One can then adapt the proof of Theorem~\ref{thm:existence-EI} to obtain a result similar to \cite[Theorem~5]{FFT10}, thus obtaining the convergence of cluster positions to a Poisson Process.  In order to achieve this, we would have to change $D^p(u_n)$ in the same way that $D_2(u_n)$ was changed to $D_3(u_n)$ in \cite{FFT10}, with exceedances in $D_3(u_n)$ from \cite{FFT10} replaced by escapes. However, to obtain the actual convergence of the Exceedance Point Processes or Hitting Times Point Processes to the compound Poisson process more work is needed since we cannot apply Kallenberg's criterion used in \cite[Theorem~5]{FFT10} because here the Poisson events are not simple, \ie they can have multiplicity.  This is studied in a work in progress.
\end{remark}

We start the proof of Theorem~\ref{thm:existence-EI} with the following two simple observations.
\begin{lemma}
  \label{lem:prob-of-union}
  For any integers $p, \ell\in\N$, $s\in \N\cup\{0\}$ and real numbers $0<u<v$ we have
  \begin{equation*}
  \sum_{j=s}^{s+\ell-1} \p(Q_{p,j}(u))\geq \p(\QQ_{p,s,\ell}^c(u))
  \geq \sum_{j=s}^{s+\ell-1} \p(Q_{p,j}(u))-
  \sum_{j=s}^{s+\ell-1}\sum_{%i\neq j,i=s
  i>j}^{s+\ell-1} \p(Q_{p,j}(u)\cap
  Q_{p,i}(u))
  \end{equation*}
\end{lemma}
\begin{proof}
This is a straightforward consequence of the formula for the
probability of a multiple union on events. See for example first
theorem of Chapter~4 in \cite{F50}.
\end{proof}
\begin{lemma}
  \label{lem:relation-maximums}
  Assume that $t,r,m,\ell,s$ are nonnegative integers and $u>0$ is a positive real number. Then, we have
  \begin{equation}\label{lem:relation-maximums-eq1}
  0\leq \p(\QQ_{p,s,\ell}(u))-\p(\QQ_{p,s,\ell+t}(u))\leq t\cdot \p(Q_{p,0}(u))
  \end{equation}
  and
  \begin{gather}
  \begin{split} \left|\p(\QQ_{p,0,s+t+m}(u))-\p(\QQ_{p,0,m}(u)) +\sum_{j=0}^{s-1} \p\left(Q_{p,0}(u)\cap
  \QQ_{p,s+t-j,m}(u)\right)\right|&\\\leq 2s\sum_{j=1}^{s-1}
  \p(Q_{p,0}(u)\cap Q_{p,j}(u))+ & t\p(Q_{p,0}(u)).
  \end{split}\label{lem:relation-maximums-eq2}
  \end{gather}
\end{lemma}

The proof of this lemma can easily be done by following the proof of \cite[Lemma~3.2]{FF08a} or \cite[Proposition~3.2]{C01} with minor adjustments.

\begin{proof}[Proof of Theorem~\ref{thm:existence-EI}]
We split the proof in two parts. The first is devoted to showing the second equality in \eqref{eq:max-ei}, leaving the first equality for the second part of the proof.

Let $\ell=\ell_n=[n/k_n]$ and $k=k_n$ be as in Condition $D_p'(u_n)$. We begin by replacing $\p(\QQ_{0,n}(u_n))$ by
$\p(\QQ_{0,k(\ell+t)}(u_n))$ for some $t>1$. By
\eqref{lem:relation-maximums-eq1} of
Lemma~\ref{lem:relation-maximums} and the fact that $Q_{p,0}(u_n)\subset\{X_0>u_n\}$, we have
\begin{equation}
\label{eq:replace} \left|\p(\QQ_{0,n}(u_n))-\p(\QQ_{0,k(\ell+t)}(u_n))\right|\leq kt\p(X_0>u_n).
\end{equation}
We now estimate recursively $\p(\QQ_{p,0,i(\ell+t)}(u_n))$ for
$i=0,\ldots, k$. Using ~\eqref{lem:relation-maximums-eq2} of
Lemma~\ref{lem:relation-maximums} and stationarity, we have for any
$1\leq i\leq k$
\begin{equation*}
  \left|\p(\QQ_{p,0,i(\ell+t)}(u_n))-\big(1-\ell \p(Q_{p,0}(u_n))\big)\p(\QQ_{p,0,(i-1)(\ell+t)}(u_n))\right|
  \leq \Gamma_{n,i},
\end{equation*}
where
\begin{align*}
  \Gamma_{n,i}&=
  \left|\ell \p(Q_{p,0}(u_n))\p(\QQ_{p,0,(i-1)(\ell+t)}(u_n))-
   \sum_{j=0}^{\ell-1} \p\left(Q_{p,j}(u_n)
  \cap\QQ_{p,\ell+t,(i-1)(\ell+t)}(u_n)\right)\right|\\
  &\qquad+t\p(X_0>u_n)+2\ell\sum_{j=1}^{\ell-1}
  \p\left(Q_{p,0}(u_n)
  \cap Q_{p,j}(u_n)\right).
\end{align*}
Using stationarity, $D^p(u_n)$ and, in particular, that $\gamma(n,t)$
is nonincreasing in $t$ for each $n$ we conclude
\begin{align*}
  \Gamma_{n,i}&\leq\sum_{j=0}^{\ell-1}
  \big|\p(Q_{p,0}(u_n))\p(\QQ_{p,0,(i-1)(\ell+t)}(u_n))-\p\left(Q_{p,0}(u_n)
  \cap \QQ_{p,\ell+t-j,(i-1)(\ell+t)}(u_n)\right)
   \big|\\
  &\quad+t\p(X_0>u_n)+2\ell\sum_{j=1}^{\ell-1}
  \p\left(Q_{p,0}(u_n)
  \cap Q_{p,j}(u_n)\right)\\
  &\leq \ell\gamma(n,t)+t\p(X_0>u_n)+2\ell\sum_{j=1}^{\ell-1}
  \p\left(Q_{p,0}(u_n)
  \cap Q_{p,j}(u_n)\right).
\end{align*}
Define
$\Upsilon_n= \ell\gamma(n,t)+t\p(X_0>u_n)+2\ell\sum_{j=1}^{\ell-1}
  \p\left(Q_{p,0}(u_n)
  \cap Q_{p,j}(u_n)\right)$. Then for every $1<i\leq k$ we have
\[
\left|\p(\QQ_{p,0,i(\ell+t)}(u_n))-\big(1-\ell \p(Q_{p,0}(u_n))\big)\p(\QQ_{p,0,(i-1)(\ell+t)}(u_n))\right|
  \leq \Upsilon_n
\]
and for $i=1$
\[
\left|\p(\QQ_{p,0,\ell+t}(u_n))-\big(1-\ell \p(Q_{p,0}(u_n))\big)\right|
  \leq \Upsilon_n.
\]

Since $n\p(X>u_n)\to\tau$, as $n\to\infty$, by \eqref{eq:level-un-annulus}, it follows that $n\p(Q_{p,0}(u_n))\to\theta\tau.$
Hence, if $k$ and $n$ are large enough we
  have $\ell \p(Q_{p,0}(u_n))<2$, which implies that
  $\big|1-\ell \p(Q_{p,0}(u_n))\big|<1$. Then, a simple inductive argument
  allows us to conclude
\begin{equation*}
  \left|\p(\QQ_{p,0,k(\ell+t)}(u_n))-\big(1-\ell \p(Q_{p,0}(u_n))\big)^k\right|\leq k\Upsilon_n.
\end{equation*}
Recalling \eqref{eq:replace}, we have
\begin{equation*}
\left|\p(\QQ_{p,0,n}(u_n))-\big(1-\ell
\p(Q_{p,0}(u_n))\big)^k\right|\leq kt\p(Q_{p,0}(u_n))+k\Upsilon_n.
\end{equation*}
Since  by \eqref{eq:level-un-annulus}, it follows that $n\p(Q_{p,0}(u_n))\to\theta\tau$, as $n\to\infty$, for some
$\tau\geq0$, we have
\[
\lim_{n\to\infty}\big(1-[\tfrac{n}{k}]\p(Q_{p,0}(u_n))\big)^k
=\e^{-\theta\tau}.
\]
It is now clear that, the second equality in \eqref{eq:max-ei} holds if \[
\lim_{n\to\infty}kt\p(Q_{p,0}(u_n))+k\Upsilon_n=0,
\]
that is
\begin{equation}
\label{eq:error}
\lim_{n\to\infty}2kt\p(Q_{p,0}(u_n))+n\gamma(n,t)+
2n\sum_{j=1}^ {\ell}
  \p\left(Q_{p,0}(u_n)
  \cap Q_{p,j}(u_n)\right)=0.
\end{equation}
Assume that $t=t_n$ where $t_n=o(n)$ is given by Condition $D^p(u_n)$.
Then, by \eqref{eq:kn-sequence}, we have \( \lim_{n\to\infty}kt_n\p(Q_{p,0}(u_n))=0
\), since $n\p(Q_{p,0}(u_n))\to\theta\tau\geq 0$. Finally, we use $D^p(u_n)$ and
$D'_p(u_n)$ to obtain that the two remaining terms in \eqref{eq:error}
also go to $0$.

Now, we need to show that the first equality in \eqref{eq:max-ei} holds. First observe that $$\{M_n\leq u_n\}\subset\QQ_{p,0,n}(u_n).$$ Next, note that if $\QQ_{p,0,n}(u_n)\setminus \{M_n\leq u_n\}$ occurs, then we may define $i=\inf\{j\in\{0,1,\ldots n-1\}:\; X_j>u_n\}$ and $s_i=[ \frac {n-1-i}{p}]$. But since $\QQ_{p,0,n}(u_n)$ does occur, then for all $j=1,\ldots,s_i$ we must have $X_{i+jp}>u_n$, otherwise, there would exist $j_i=\min\{j\in\{1,\ldots,s_i\}: X_j\leq u_n\}$ and $Q_{p,i+(j_i-1)p}(u_n)$ would occur, which contradicts the occurrence of $\QQ_{p,0,n}(u_n)$. This means that
\[
\QQ_{p,0,n}(u_n)\setminus \{M_n\leq u_n\}\subset \bigcup_{i=0}^{n-1} \{X_i>u_n,X_{i+p}>u_n,\ldots, X_{i+s_ip}>u_n\}.
\]
It follows by $\spp(u_n)$ and stationarity that
\begin{multline*}
  \p(\QQ_{p,0,n}(u_n))-\p(\{M_n\leq u_n\})\leq \sum_{i=0}^{n-1}\p\left( X_i>u_n,X_{i+p}>u_n,\ldots, X_{i+s_ip}>u_n \right)\\
    \leq p\sum_{i=0}^{[n/p]}\p\left(X_0>u_n, X_{p}>u_n,X_{2p}>u_n,\ldots, X_{ip}>u_n\right)\xrightarrow[n\to\infty]{}0.
  \end{multline*}
\end{proof}

\subsection{Existence of an EI due to multiple underlying periodic phenomena}
\label{ssec:higher ord}
In this subsection we consider  stochastic processes with more than one underlying periodic phenomenon creating clustering of events (these can not be realised as stochastic processes coming from dynamical systems as described above).

In fact, it may happen that the escapes themselves form clusters which means that $D_p'(u_n)$ does not hold. This occurs if, for example, for some $1\leq j\leq[n/k_n]$ we have $n\p(Q_{p,0}(u_n)\cap Q_{p,j}(u_n))\to\alpha>0$. Let $p_2$ be the smallest such $j$.  Then, since $\p(Q_{p,0}(u_n))/n\sim \theta\tau$, we  have that \eqref{cond:periodicity} holds if we replace exceedances by escapes and $p$ by $p_2$. Therefore there is a second underlying periodic phenomenon which leads to the notion of escapes of second order.  This motivates the introduction of similar conditions to $SP_{p,\theta}$, $D^p(u_n)$, $D_p'(u_n)$, where the role of the exceedances is replaced by escapes, in order to obtain a statement like Theorem~\ref{thm:existence-EI}, where the distribution of the maximum would be equal to the distribution of these escapes of second order. Since it may also happen that these escapes of second order also form clusters, we may have to repeat the process all over again. Hence, we establish a hierarchy of escapes in the following way.

Given the sequences $(p_i)_{i\in\N}$ and $(\theta_i)_{i\in\N}$, with $p_i\in\N$ and $\theta_i\in(0,1)$ for all $i\in\N$, let $\mathbf{p}_i=(p_1,p_2,\ldots,p_i)$, $\Theta_i=(\theta_1,\theta_2,\ldots,\theta_i)$. For each $j\in\N$ and $u\in\R$,  assuming that $Q^{(i-1)}_{\mathbf p_{i-1},j}(u)$ is already defined we define the \emph{escape of order $i$} as
\begin{equation*}
Q^{(i)}_{\mathbf p_{i},j}(u)=Q^{(i-1)}_{\mathbf p_{i-1},j}(u)\cap \left(Q^{(i-1)}_{\mathbf p_{i-1},j+p_i}(u)\right)^c.
\end{equation*}
We set $Q^{(1)}_{\mathbf p_{1},j}(u)=Q_{p_1,j}(u)$ and in the case $i=0$ we can consider that $Q^{(0)}_{0,j}(u)=\{X_j>u\}$. For $i,s,\ell\in\N\cup\{0\}$, let $\QQ^{(i)}_{\mathbf p_i,s,\ell}(u)=\bigcap_{j=s}^{s+\ell-1} \left(Q^{(i)}_{\mathbf p_i,j}(u)\right)^c.$
Now we restate conditions $SP_{p,\theta}$, $D^p(u_n)$, $D_p'(u_n)$ with respect to the escapes of order $i\in\N$.
\begin{condition}[$\sppi(u_n)$]\label{cond:SPi} We say that  $X_0,X_1,X_2,\ldots$ satisfies condition $\sppi(u_n)$, for $\mathbf p_i\in \N^i$ and $\Theta_i\in (0,1)^i$ defined as above, if
\begin{equation}
\label{cond:i-periodicity}
\begin{split}\lim_{n\to \infty}\sup_{1\le j<p_i}\p\left(Q^{(i-1)}_{\mathbf p_{i-1},j}(u_n)\big| Q^{(i-1)}_{\mathbf p_{i-1},0}(u_n)\right)=0 \;\mbox{ and}\hspace{5cm}\\ \lim_{n\to \infty} \p\left(Q^{(i-1)}_{\mathbf p_{i-1},p_i}(u_n)\big| Q^{(i-1)}_{\mathbf p_{i-1},0}(u_n)\right)\to (1-\theta_i)
\end{split}
\end{equation}
and moreover
\begin{equation}
\label{cond:i-summability}
\lim_{n\to \infty}\sum_{j=0}^{\left[\frac{n-1}{p_i}\right]} \p\left(Q^{(i-1)}_{\mathbf p_{i-1},0}(u_n), Q^{(i-1)}_{\mathbf p_{i-1},p_i}(u_n), Q^{(i-1)}_{\mathbf p_{i-1},2p_i}(u_n),\ldots,Q^{(i-1)}_{\mathbf p_{i-1},j p_i}(u_n)\right)=0.
\end{equation}
\end{condition}
\begin{condition}[$D^{\mathbf p_i}(u_n)$]\label{cond:Dpi}We say that $D^{\mathbf p_i}(u_n)$ holds
for the sequence $X_0,X_1,X_2,\ldots$ if for any integers $\ell,t$
and $n$
\[ \left|\p\left(Q^{(i)}_{\mathbf p_{i},0}(u_n)\cap
  \QQ^{(i)}_{\mathbf p_i,t,\ell}(u_n)\right)-\p\left(Q^{(i)}_{\mathbf p_{i},j}(u_n)\right)
  \p\left(\QQ^{(i)}_{\mathbf p_i,0,\ell}(u_n)\right)\right|\leq \gamma(n,t),
\]
where $\gamma(n,t)$ is nonincreasing in $t$ for each $n$ and
$n\gamma(n,t_n)\to0$ as $n\rightarrow\infty$ for some sequence
$t_n=o(n)$.
\end{condition}
Note that, as before, the first event $Q^{(i)}_{\mathbf p_{i},0}(u_n)$ depends only on a finite number of r.v., namely, $X_0, X_{p_1}, X_{p_1+p_2}, \ldots, X_{p_1+p_2+\cdots+p_i}$.

\begin{condition}[$D'_{\mathbf p_{i}}(u_n)$]\label{cond:D'pi} We say that $D'_{\mathbf p_{i}}(u_n)$
holds for the sequence $X_0,X_1,X_2,\ldots$ if there exists a sequence $\{k_n\}_{n\in\N}$ satisfying \eqref{eq:kn-sequence} and such that
\begin{equation*}
\lim_{n\rightarrow\infty}\,n\sum_{j=1}^{[n/k_n]}\p\left( Q^{(i)}_{\mathbf p_{i},0}(u_n)\cap
Q^{(i)}_{\mathbf p_{i},j}(u_n)\right)=0.
\end{equation*}
\end{condition}

Observe that condition $D'_{\mathbf p_{i}}(u_n)$ gets weaker and weaker as $i$ increases, which means that every time a new underlying periodic phenomenon is found, there is a higher chance that escapes of the next order satisfy $D'$.

The next result generalises Theorem~\ref{thm:existence-EI}, which corresponds exactly to the case $i=1$, to the case of higher order escapes. We stated these theorems separately since Theorem~\ref{thm:existence-EI} contains the essential ideas required for Theorem~\ref{thm:existence-EI-i} and, moreover, shows the influence of the periodic behaviour in a more transparent way.

\begin{theorem}
  \label{thm:existence-EI-i}
  Let $(u_n)_{n\in\N}$ be such that $n\p(X>u_n)=n(1-F(u_n))\to\tau$,
  as $n\to\infty$, for some $\tau\geq 0$.
  Consider a stationary stochastic process $X_0, X_1, X_2,\ldots$ satisfying conditions $\sppj(u_n)$ for all $1\leq j\leq i$.
  Assume further that conditions
  $D^{\mathbf p_i}(u_n)$ and $D'_{\mathbf p_i}(u_n)$ hold. Then
  \begin{equation}
  \label{eq:max-ei-i}
   \lim_{n\to\infty}\p(M_n\leq u_n)=\lim_{n\to\infty}\p\left(\QQ^{(1)}_{p,0,n}(u_n)\right) =\cdots=\lim_{n\to\infty}\p\left(\QQ^{(i)}_{p,0,n}(u_n)\right) =\e^{-\theta \tau},
  \end{equation}
  where $\theta=\theta_1\theta_2\cdots \theta_i$.
\end{theorem}

We notice that in the particular case $i=2$ then $\theta_2$ corresponds to the upcrossings index $\eta$ in \cite{F06}.

\begin{proof}
The proof of the last equality in \eqref{eq:max-ei-i} is basically done as in Theorem~\ref{thm:existence-EI} simply by replacing everything by its corresponding $i$ version.

The proof of the $j$-th equality, with $1\leq j\leq i$, in \eqref{eq:max-ei-i} follows as the proof of the first equality in \eqref{eq:max-ei} except that instead of \eqref{cond:summability}  we use \eqref{cond:i-summability} of the corresponding condition $\sppj$.

Regarding the formula for the extremal index $\theta$ observe that it follows by an easy induction argument from the fact that $\sppj(u_n)$ holds for all $1\leq j\leq i$. In fact, as always, let $(u_n)_{n\in\N}$ be a sequence of levels such that $n(1-F(u_n))=n\p(X_0>u_n)\to \tau$, as $n\to\infty$, for some $\tau\geq0$. Assuming by induction that $n\p(Q_{\mathbf p_{j-1},0}^{(j-1)}(u_n))\to \theta_1\theta_2\ldots\theta_{j-1} \tau$, as $n\to\infty$, by \eqref{cond:i-periodicity}  of condition $\sppj(u_n)$ we must have:
\begin{align*}
n\p\left(Q_{\mathbf p_{j},0}^{(j)}(u_n)\right)&=n\p\left(Q_{\mathbf p_{j-1},0}^{(j-1)}(u_n)\cap \left(Q_{\mathbf p_{j-1},p_{j}}^{(j-1)}(u_n)\right)^c\right)\\
&=n\p\left(Q_{\mathbf p_{j-1},0}^{(j-1)}(u_n)\right)
\p\left(\left(Q_{\mathbf p_{j-1},p_{j}}^{(j-1)}(u_n)\right)^c \Big| Q_{\mathbf p_{j-1},0}^{(j-1)}(u_n)\right)\xrightarrow[n\to\infty]{} \theta_1\theta_2\ldots\theta_{j-1} \tau \theta_j.
\end{align*}
Since, by \eqref{eq:level-un-annulus},  we have $n\p(Q_{\mathbf p_{1},0}^{(1)}(u_n))\to \theta_1 \tau$, as $n\to\infty$, the result follows at once.
\end{proof}

When comparing Theorem~\ref{thm:existence-EI-i} with similar results in the literature, particularly the most similar in \cite{LN89, CHM91} and \cite{F06}, we highlight the following advantages: the interpretation of the EI is explicitly motivated by the existence of underlying periodic phenomena; and the fact that our conditions are weaker, especially because our condition $D^{\mathbf p_i}(u_n)$ is much weaker than $D(u_n)$.  In fact, as we explain in greater depth in Section~\ref{subsubsec:roles-Dp-D'p}, if we had to check $D(u_n)$ for stochastic processes arising from dynamical systems we could only get HTS/RTS for cylinders (see definition in Section~\ref{sec:cyl}) instead of balls, which we do obtain in Corollaries~\ref{cor:HTS/RTS-EI-abs-cont} and \ref{cor:HTS/RTS-EI-eq-st}. In terms of EVL, this means that we would get cylinder EVL with convergence only for certain subsequences $\omega_n$ of time $n\in\N$, which contrasts with our results in Theorems~\ref{thm:EI-abs-cont} and \ref{thm:EI-eq-st}.

Regarding applications of Theorems~\ref{thm:existence-EI} and \ref{thm:existence-EI-i}, we mention that for the examples of stochastic processes that besides $D(u_n)$ also satisfy $D''(u_n)$ from \cite{LN89}, then Theorem~\ref{thm:existence-EI} can be used to prove the existence of an EI. While for the examples we know of stochastic processes that, for some $k\geq 2$, satisfy $D^{(k)}(u_n)$ from \cite{CHM91} instead, then eventually Theorem~\ref{thm:existence-EI-i} can be used for the same purpose.

Besides the applications to stochastic processes coming from dynamical systems given in Section~\ref{sec:EI-DS}, for which $\mpp$ is shown to hold, we give two examples in the appendix, one of Maximum Moving Average sequences and one of an Autoregressive process, to which the results of this section also apply. While these examples are not novel, they illustrate how to check conditions $\spp$, $\mpp$, $D^p(u_n)$, $D'_p(u_n)$, $\sppi$, $D^{\mathbf p_i}(u_n)$ and $D'_{\mathbf p_i}(u_n)$ in different, more classical, settings. The Maximum Moving Average in Appendix~\ref{sec:MMA-101} satisfies $\spp$ with $p=2$, and $\theta=1/2$, while the one in Appendix~\ref{sec:MMA-1101} satisfies $\sppi$, with $i=1,2$, $\mathbf p_2=(p_1,p_2)=(1,3)$, $\Theta_2=(2/3,1/2)$. The Autoregressive process of order 1 (AR(1)), introduced in \cite{C81} and considered in Appendix~\ref{sec:AR1}, is shown to satisfy $\mpp$ with $p=1$ and $\theta=1-1/r$.

\section{The general theory for sequences generated by dynamical systems}
\label{sec:EI-DS}

In this section, we set out the general theory of the extremal index in the context of a discrete time dynamical system $(\X,\mathcal
B,\mu,f)$, where $\X$ is a Riemannian manifold, $\mathcal B$ is the Borel $\sigma$-algebra, $f:\X\to\X$ is a measurable map
and $\mu$ an $f$-invariant probability measure.  We will initially show that $\mpp(u_n)$ can be proved for quite general systems, and later, in Section~\ref{sec:dyn egs}, give specific examples where we can also prove $D_p(u_n)$ and $D_p'(u_n)$ and thus apply Theorem~\ref{thm:existence-EI}.

We consider a Riemannian metric on $\X$
that we denote by `$\dist$' and for any $\zeta\in\X$ we define the ball of radius $\delta>0$ around $\zeta$, as $B_{\delta}(\zeta)=\{x\in\X: \dist(x,\zeta)<\delta\}$. Let $\l$ denote a normalised volume form defined on $\B$ that we call Lebesgue measure.

We suppose that the stochastic processes $X_0, X_1, X_2,\ldots$ defined by \eqref{eq:def-stat-stoch-proc-DS} are such that the r.v. $\varphi:\X\to\R\cup\{\pm\infty\}$
achieves a global maximum at $\zeta\in \X$ (we allow
$\varphi(\zeta)=+\infty$).

In order to study the statistical properties of the system, the invariant probability measure $\mu$ and its properties play a crucial role. First, we want the measure to provide relevant information about the system.
This is achieved, for example, by requiring that the measure is `physical' or even more generally an `equilibrium state'.  We will emphasise the first kind of measures here due to their importance in the study of the statistical properties of dynamical systems.

A measure $\mu$ is said to be \emph{physical} if the Lebesgue measure of the set of points $\U$ (called the basin of $\mu$), for which the law of large numbers holds for any stochastic process defined as in \eqref{eq:def-stat-stoch-proc-DS} for any continuous r.v. $\psi:\X\to\R$, is positive. In other words, if the set of points $x$ such that
\begin{equation}
\label{eq:ergodic-averages}
\frac 1n\sum_{i=0}^{n-1} \psi (f^i(x))\to \int \psi d\mu
\end{equation}
has positive Lebesgue measure. For example, $\mu$ is a physical measure if it is absolutely continuous with respect to Lebesgue, in which case we write $\mu\ll\l$, and \emph{ergodic}, which simply means that \eqref{eq:ergodic-averages} holds $\mu$-a.e. Note that these measures do provide a nice picture of the statistical behaviour of the system, since describing how the time averages $\frac 1n\sum_{i=0}^{n-1} \psi (f^i(x))$  of any continuous function $\psi$ behave, reduces to compute the spacial average $\int \psi d\mu$ simply by integrating $\psi$ against the measure $\mu$. Moreover, this works on a ``physically observable'' set $\U$ of positive Lebesgue measure.

More generally, we can study the statistical properties of a system through the following class of measures, known as equilibrium states.  For good introductions to this topic see for example \cite{Bo75, W82, K98}.

Let $f:\X \to \X$ be a measurable function as above, and define 
$$\M_f:=\left\{f-\text{invariant Borel probability measures on } \X\right\}.$$
I.e., for $\mu\in \M_f$, $\mu(\X)=1$ and for any Borel measurable set $A$, $\mu(f^{-1}(A))=\mu(A)$.
Then for a measurable potential $\phi:\X\to \R$, we define the \emph{pressure} of $(\X,f,\phi)$ to be
$$P(\phi):=\sup_{\mu\in \M_f}\left\{h(\mu)+\int\phi~d\mu:-\int\phi~d\mu<\infty\right\},$$
where $h(\mu)$ denotes the metric entropy of the measure $\mu$, see \cite{W82} for details.
If, for $\mu\in \M_f$, $h(\mu)+\int\phi~d\mu=P(\phi)$ then we say that $\mu$ is an \emph{equilibrium state} for $(\X, f, \phi)$.

The absolutely continuous measures given above can often be shown to be particular examples of equilibrium states.  This is  explained in more depth in Section~\ref{ssec:eq setup}.

\subsection{Measures absolutely continuous with respect to Lebesgue}
\label{subsec:abs-cont}

In this subsection, we assume that the measure $\mu$ is absolutely continuous with respect to Lebesgue. Besides, we assume  that $\zeta$ is a \emph{repelling $p$-periodic point}, which means that $f^p(\zeta)=\zeta$, $f^p$ is differentiable at $\zeta$ and $0<\left|\det D(f^{-p})(\zeta)\right|<1$.  Moreover, we also assume that $\zeta$ is a Lebesgue density point with $0<\frac{d\mu}{d\l}(\zeta)<\infty$ and the observable $\varphi:\X\to\R\cup\{+\infty\}$ is of
the form
\begin{equation}
\label{eq:observable-form} \varphi(x)=g(\dist(x,\zeta)),
\end{equation} where
the function $g:[0,+\infty)\rightarrow {\mathbb
R\cup\{+\infty\}}$ is such that $0$ is a global maximum ($g(0)$ may
be $+\infty$); $g$ is a strictly decreasing bijection $g:V \to W$
in a neighbourhood $V$ of
$0$; and has one of the
following three types of behaviour:
\begin{enumerate}[Type 1:]
\item there exists some strictly positive function
$\hat\pi:W\to\R$ such that for all $y\in\R$
\begin{equation}\label{eq:def-g1}\displaystyle \lim_{s\to
g_1(0)}\frac{g_1^{-1}(s+y\hat\pi(s))}{g_1^{-1}(s)}=\e^{-y};
\end{equation}
\item $g_2(0)=+\infty$ and there exists $\beta>0$ such that
for all $y>0$
\begin{equation}\label{eq:def-g2}\displaystyle \lim_{s\to+\infty}
\frac{g_2^{-1}(sy)}{g_2^{-1}(s)}=y^{-\beta};\end{equation}
\item $g_3(0)=D<+\infty$ and there exists $\gamma>0$ such
that for all $y>0$
\begin{equation}\label{eq:def-g3}\lim_{s\to0}
\frac{g_3^{-1}(D-sy)}{g_3^{-1}(D-s)}= y^\gamma.
\end{equation}
\end{enumerate}

Examples of each one of the three types are as follows:
$g_1(x)=-\log x$ (in this case \eqref{eq:def-g1} is easily verified
with $\hat\pi\equiv1$), $g_2(x)=x^{-1/\alpha}$ for some $\alpha>0$ (condition
\eqref{eq:def-g2} is verified with $\beta=\alpha$) and
$g_3(x)=D-x^{1/\alpha}$ for some $D\in\R$ and $\alpha> 0$ (condition
\eqref{eq:def-g3} is verified with $\gamma=\alpha$).

\begin{remark}
\label{rem:attraction-domains}
  Recall that the d.f. $F$ is given by
$F(u)=\mu(X_0\leq u)$ and $u_F=\sup\{y: F(y)<1\}$. Observe that
if at time $j\in\N$ we have an exceedance of the level $u$
(sufficiently large), i.e., $X_j(x)>u$, then we have an entrance of
the orbit of $x$ into the ball $B_{g^{-1}(u)}(\zeta)$ of radius
$g^{-1}(u)$ around $\zeta$, at time $j$. This means that the
behaviour of the tail of $F$, \ie the behaviour of $1-F(u)$ as $u\to
u_F$ is determined by $g^{-1}$, if we assume that Lebesgue's
Differentiation Theorem holds for $\zeta$, since in that case
$1-F(u)\sim \rho(\zeta) |B_{g^{-1}(u)}(\zeta)|$, where
$\rho(\zeta)=\frac{d\mu}{d\l}(\zeta)$. From classical Extreme Value
Theory we know that the behaviour of the tail determines the limit
law for partial maximums of i.i.d. sequences and vice-versa. The
above conditions are just the translation in terms of the shape of
$g^{-1}$, of the sufficient and necessary conditions on the tail of
$F$ of \cite[Theorem~1.6.2]{LLR83}, in order to exist a
non-degenerate limit distribution for $\hat M_n$.
\end{remark}

Recall that $X_0, X_1, X_2,\ldots$ is given by \eqref{eq:def-stat-stoch-proc-DS} for
observables of the type \eqref{eq:observable-form}, which means the event $\left\{X_0>u\right\}$ corresponds to a ball centred
at $\zeta$. Suppose that $p\in \N$ and consider as before
$$Q_{p,0}^*(u):=\{x:\;\varphi(x)>u\}\cap f^{-p}(\{x:\;\varphi(x)>u\})\text{ and } Q_{p,0}(u):=\{x:\;\varphi(x)>u\}\setminus Q^*(u).$$
The set  $Q_{p}^*(u):=Q_{p,0}^*(u)$ corresponds to a ball, while $Q_{p}(u):=Q_{p,0}(u)$ corresponds to an annulus, both centred at $\zeta$. For all $i\in\N$, set
\[
Q_p^{*i}(u):=\bigcap_{j=0}^{i-1} Q_{p,j}^*(u).
\]

Given the special structure of these dynamically defined stochastic processes, observe that for all $i\in\N$ we have $Q_{p,i}(u)=f^{-i}(Q_p(u))=\{X_i\in Q_{p}(u)\}$, $Q_{p,i}^*(u)=f^{-i}(Q_p^*(u))=\{X_i\in Q_{p}^*(u)\}$ and 
$Q_p^{*i}(u)=\bigcap_{j=0}^{i} f^{-jp}(\{x:\,\varphi(x)>u\})$.

We will provide some conditions which guarantee an Extreme Value Law with a given extremal index.  We will give some systems which satisfy these conditions in Section~\ref{sec:dyn egs}.

\begin{theorem}
  \label{thm:EI-abs-cont}
  Suppose that $\zeta$ is a repelling periodic point of prime period $p$, with $\theta=\theta(\zeta)=1-|\det D(f^{-p})(\zeta)|\in (0,1)$.  Let $(u_n)_{n\in\N}$ be such that $n\mu(X_0>u_n)=n(1-F(u_n))\to\tau$,
  as $n\to\infty$, for some $\tau\geq 0$.
  Assume further that conditions
  $D^p(u_n)$ and $D'_p(u_n)$ hold. Then
  \begin{equation*}
   \lim_{n\to\infty}\mu(M_n\leq u_n)=\lim_{n\to\infty}\mu(\QQ_{p,0,n}(u_n))=\e^{-\theta \tau}.
  \end{equation*}
\end{theorem}

\begin{proof}
To prove Theorem~\ref{thm:EI-abs-cont} we only need to show property $\mpp(u_n)$ and apply Theorem~\ref{thm:existence-EI}. Since $\zeta$ is a repelling periodic point, by the Mean Value Theorem we have
$$\l(Q^*(u))\sim(1-\theta)\l(\{x:\varphi(x)>u\}) $$
for $u$ close to $u_F$. By induction, we get
\[
\l(Q_p^{*i}(u))=\l\left(\bigcap_{j=0}^{i} f^{-jp}(\{x:\,\varphi(x)>u\})\right)\sim (1-\theta)^i \l(\{x:\varphi(x)>u\}),
\]
for $u$ close to $u_F$. Consequently, using the fact that $\zeta$ is a Lebesgue density point, we have for $i\in\N$,
\begin{align*}
\mu(X_p>u,\ldots,X_{ip}>u|X_0>u)&=\frac{\mu(Q_p^{*i}(u))}{\mu(X_0>u)} \sim \frac{\l(Q_p^{*i}(u))}{\l(\{x:\phi(x)>u\})}\\
	&\sim (1-\theta)^i.
\end{align*}
So replacing $u$ with $(u_n)_n$, summing over $i$ and letting $n\to \infty$, we have $\mpp$, as required.
\end{proof}

The relation between EVL and HTS established in \cite{FFT10} allows us to obtain the following:
\begin{corollary}
\label{cor:HTS/RTS-EI-abs-cont}
Suppose that $\zeta$ is a repelling periodic point of prime period $p$, with $\theta=\theta(\zeta)=1-|\det D(f^{-p})(\zeta)|\in (0,1)$.  Let $(u_n)_{n\in\N}$ be such that $n\mu(X_0>u_n)=n(1-F(u_n))\to\tau$,
  as $n\to\infty$, for some $\tau\geq 0$.
  Assume further that conditions
  $D^p(u_n)$ and $D'_p(u_n)$ hold. Then we have Hitting Time Statistics to balls at $\zeta$,
  \begin{equation}
  \label{eq:HTS-ei}
   \lim_{n\to\infty}\mu\left(r_{B_{\delta_n}(\zeta)}< \frac t{\mu(B_{\delta_n}(\zeta))}\right)=1-\e^{-\theta t},
  \end{equation}
  and Return Time Statistics to balls at $\zeta$,
   \begin{equation}
  \label{eq:RTS-ei}
   \lim_{n\to\infty}\mu_{B_{\delta_n}(\zeta)}\left(r_{B_{\delta_n}(\zeta)}< \frac t{\mu(B_{\delta_n}(\zeta))}\right)=(1-\theta)+\theta(1-\e^{-\theta t}),
  \end{equation}
  for all sequences $\delta_n\to0$, as $\n\to\infty$.
\end{corollary}

\begin{proof}
    The limit \eqref{eq:HTS-ei} is an immediate consequence of Theorem~\ref{thm:EI-abs-cont} and \cite[Theorem~2]{FFT10}. The limit  \eqref{eq:RTS-ei} derives from \eqref{eq:HTS-ei} and the integral formula in \cite[Main~Theorem]{HLV05} that relates the HTS and RTS distributions.
\end{proof}

\subsection{Equilibrium states}
\label{ssec:eq setup}

We gave the notions of pressure and equilibrium states at the beginning of this section.  Here we will introduce further notions in order to generalise Theorem~\ref{thm:EI-abs-cont} to general equilibrium states.  A measure $m$ is called a \emph{$\phi$-conformal} measure if $m(\X)=1$ and if whenever $f:A\to f(A)$ is a bijection, for a Borel set $A$, then
$$m(f(A))=\int_A e^{-\phi}~dm.$$

Note that for example for a smooth map interval map $f$, Lebesgue measure is $\phi$-conformal for $\phi(x):=-\log|Df(x)|$.  Moreover, if for example $f$ is a topologically transitive quadratic interval map then as in Ledrappier \cite{Le81}, any physical measure $\mu$ with $h(\mu)>0$ is an equilibrium state for $\phi$.  That this also holds for the even simpler case of piecewise smooth uniformly expanding maps follows from Section~\ref{ssec:Rychlik}.

We define
$$S_n\phi(x):=\phi(x)+\cdots +\phi\circ f^{n-1}(x).$$
In the following proposition we will assume that for a potential $\phi$, we have $P(\phi)=0$.  Note that if $P(\phi)=p$ for $p\in (-\infty,\infty)$ then we can  replace $\phi$ by $\phi-p$ to obtain $P(\phi-p)=0$.  Clearly any equilibrium state for $\phi$ is an equilibrium state for $\phi-p$ and vice versa.  Recall that we are assuming that $f:\X\to \X$ is a measurable map of a Riemannian manifold which is differentiable at a periodic point $\zeta\in \X$.

Again, we consider that the stochastic processes $X_0, X_1, X_2,\ldots$ defined by \eqref{eq:def-stat-stoch-proc-DS} are such that the r.v. $\varphi:\X\to\R\cup\{\pm\infty\}$
achieves a global maximum at $\zeta\in \X$. However, in order to still be able to establish the connection between EVL and HTS in this setting, where the invariant measure may present a more irregular behaviour than when it is absolutely continuous, we need to tailor the observable $\varphi$ to cope with this lack of regularity as in  \cite{FFT11}. Essentially, this means that we need to replace $\dist(x, \zeta)$  in \eqref{eq:observable-form} with $\mu(B_{\dist(x, \zeta)}(\zeta))$. So throughout this section the stochastic processes $X_0, X_1, X_2,\ldots$ defined by \eqref{eq:def-stat-stoch-proc-DS} is such that the r.v. $\varphi:\X\to\R\cup\{\pm\infty\}$ is given by
\begin{equation}
\label{eq:observable-form-eq-st} \varphi(x)=g(\mu(B_{\dist(x, \zeta)}(\zeta))),
\end{equation}
where $g$ is as in Section~\ref{subsec:abs-cont}.

Since for this application, we would like a sequence $(u_n)_n$ such that $\lim_{n\to\infty}n\mu(\{X_0>u_n\}) =\tau$, it is useful to assume that the measure $\mu$ has some continuity: otherwise, `jumps' in the size of balls around $\zeta$ may prevent us from finding such a sequence.  To deal with this issue, in \cite{FFT11}, we defined a function $\size$ for small $\eta\geq 0$ and given by
\begin{equation}
\label{eq:def-h}
\size(\eta)=\mu(B_\eta (\zeta)).
\end{equation}
We required that $\size$ is continuous on $\eta$. For example, if $\X$ is an interval and $\mu$ a Borel probability with no atoms,\ie points with positive $\mu$ measure, then $\size$ is continuous.

\begin{theorem}
Let $\phi:\X\to \R$ be a continuous potential with $P(\phi)=0$, a conformal measure $m_\phi$ and an equilibrium state $\mu_\phi\ll m_\phi$.  Suppose that the observable $\varphi$ is as in \eqref{eq:observable-form-eq-st} and that $\zeta\in \X$ is a repelling periodic point of prime period $p$ such that $0<\frac{d\mu_\phi}{dm_\phi}(\zeta)<\infty$ and
for all $u$ sufficiently close to $u_F$, \begin{equation}
\sum_{j=0}^\infty \sup\{|S_p\phi(x)-S_p\phi(\zeta)|:x\in \{\varphi>u\}\cap f^{-jp}(\{\varphi>u\})\}<\infty.\label{eq:cty phi}
\end{equation}
Let $(u_n)_{n\in\N}$ be such that $n\mu(X_0>u_n)=n(1-F(u_n))\to\tau$,
  as $n\to\infty$, for some $\tau\geq 0$.
  Assume further that conditions
  $D^p(u_n)$ and $D'_p(u_n)$ hold. Then
  \begin{equation*}
   \lim_{n\to\infty}\mu(M_n\leq u_n)=\lim_{n\to\infty}\mu(\QQ_{p,0,n}(u_n))=\e^{-\theta \tau},
  \end{equation*}
where $\theta=1-e^{S_p\phi(\zeta)}$.
\label{thm:EI-eq-st}
\end{theorem}

Condition \eqref{eq:cty phi} is to control the distortion of $\phi$ on small scales.  It follows for example from a H\"older condition on $\phi$.

\begin{remark}
If $\zeta\in \X$ is a repelling $p$-periodic point in the support of a $\phi$-conformal measure $m_\phi$ as in Theorem~\ref{thm:EI-eq-st}, then $S_p\phi(\zeta)$ must be non-positive.  We can show this by taking a very small set $A$ around $\zeta$ such that $f^p:A \to f^p(A)$ is a bijection and such that $A \subset f^p(A)$, then
$$m_\phi(f^p(A))=\int_Ae^{-S_p\phi}~dm_\phi\approx e^{-S_p\phi(\zeta)}m_\phi(A).$$  So if $S_p\phi(\zeta)>0$ then $m_\phi(f^p(A))<m_\phi(A)$, which is impossible.
\label{rmk:neg phi}
\end{remark}

Using the theory developed in \cite{FFT11}, an analogue of Corollary~\ref{cor:HTS/RTS-EI-abs-cont} holds for equilibrium states, namely:

\begin{corollary}
\label{cor:HTS/RTS-EI-eq-st}
Under the conditions of Theorem~\ref{thm:EI-eq-st} we have Hitting Time Statistics to balls at $\zeta$,
  \begin{equation*}
   \lim_{n\to\infty}\mu\left(r_{B_{\delta_n}(\zeta)}< \frac t{\mu(B_{\delta_n}(\zeta))}\right)=1-\e^{-\theta t},
  \end{equation*}
  and Return Time Statistics to balls at $\zeta$,
   \begin{equation*}
   \lim_{n\to\infty}\mu_{B_{\delta_n}(\zeta)}\left(r_{B_{\delta_n}(\zeta)}< \frac t{\mu(B_{\delta_n}(\zeta))}\right)=(1-\theta)+\theta(1-\e^{-\theta t}),
  \end{equation*}
  for all sequences $\delta_n\to0$, as $\n\to\infty$.
\end{corollary}

The proof of Theorem~\ref{thm:EI-eq-st} follows almost immediately from the following lemma.

\begin{lemma}
Let $\phi:\X\to \R$ be a potential which is continuous at $\zeta, f(\zeta), \ldots, f^{p-1}(\zeta)$ as in Theorem~\ref{thm:EI-eq-st} with $P(\phi)=0$.  If $\phi$ has a conformal measure $m_\phi$ then $\frac{m_\phi(Q_p^*(u))}{m_\phi(\{\phi>u\})} \to e^{S_p\phi(\zeta)}$ as $u\to u_F$.
Moreover, if \eqref{eq:cty phi} holds
then $$m_\phi(Q_p^{* i}(u))\sim m_\phi(\{\varphi>u\}) e^{S_{ip}\phi(\zeta)} =m_\phi(\{\varphi>u\}) (1-\theta)^i$$
for  $\theta=1-e^{S_p\phi(\zeta)}$.
\label{lem:conformal}
\end{lemma}

\begin{proof}
The continuity of $\phi$ implies that for any $\eps>0$, for all $u$ sufficiently close to $u_F$, $e^{|S_p\phi(x)-S_p\phi(y)|}< (1+\eps)$ for $x,y\in \{\varphi>u\}$.  Using this and conformality, for $u$ close enough to $u_F$,
\begin{align*}
\frac{m_\phi\left(Q_p^*(u)\right)}{m_\phi(\{\varphi>u\})}& = \frac{m_\phi\left(\{\varphi>u\}\cap f^{-p}(\{\varphi>u\})\right)}{m_\phi(\{\varphi>u\})} \sim\frac{m_\phi\left(f^p(\{\varphi>u\})\cap \{\varphi>u\}\right)}{m_\phi(f^p\{\varphi>u\})}\\
& = \frac{m_\phi\left(\{\varphi>u\}\right)}{m_\phi(f^p\{\varphi>u\})}= \frac{m_\phi\left(\{\varphi>u\}\right)}{\int_{\{\varphi>u\}}e^{-S_p\phi}~dm_\phi} \sim e^{S_p\phi(\zeta)},
\end{align*}
proving the first part of the lemma.

For the second part of the lemma, note that \eqref{eq:cty phi} implies that for any $\eps>0$ we can choose $u$ so close to $u_F$ that
$$\sum_{j=0}^\infty \sup\{|S_p\phi(x)-S_p\phi(\zeta)|:x\in f^{-jp}(\{\varphi>u\})\}<\eps.$$
So as in the proof of Theorem~\ref{thm:EI-abs-cont}, we inductively obtain
\begin{align*}
m_\phi(Q_p^{*i}(u))&= m_\phi\left(\bigcap_{j=0}^i f^{-jp}(\{\varphi>u\})\right)
\sim e^{S_{ip}\phi(\zeta)}m_\phi(\{\varphi>u\}) 
\end{align*}
as required.
\end{proof}

\begin{proof}[Proof of Theorem~\ref{thm:EI-eq-st}]
The proof is reduced to applying Theorem~\ref{thm:existence-EI} after checking that $\mpp(u_n)$ holds with $\theta=1-e^{S_p\phi(\zeta)}$ for any sequence $(u_n)_n$ with $u_n\to u_F$ as $n\to \infty$, which follows using Lemma~\ref{lem:conformal} and the same ideas as those in the proof of Theorem~\ref{thm:EI-abs-cont}.
\end{proof}

Before giving specific examples of dynamical systems satisfying $D_p(u_n)$ and $D'_p(u_n)$, in the next subsection we discuss general conditions which imply those conditions.

\subsection{On the roles of $D^p(u_n)$ and $D'_p(u_n)$ for stochastic processes arising from dynamical systems}
\label{subsubsec:roles-Dp-D'p}

Theorems~\ref{thm:EI-abs-cont} and \ref{thm:EI-eq-st} and Corollaries~\ref{cor:HTS/RTS-EI-abs-cont} and \ref{cor:HTS/RTS-EI-eq-st} assert that the existence of limiting laws of rare events with an extremal index for stochastic processes arising from dynamical systems as in \eqref{eq:def-stat-stoch-proc-DS} for observables given by \eqref{eq:observable-form} or \eqref{eq:observable-form-eq-st} centred at repelling periodic points depends on the good mixing properties of the system both at long range ($D^p(u_n)$) and short range ($D_p'(u_n)$).

In general terms the Condition $D^p(u_n)$ follows from sufficiently fast (e.g. polynomial) decay of correlations of the dynamical system. This is where $D^p(u_n)$ are seen to be much more useful than Leadbetter's $D(u_n)$. While $D(u_n)$ usually follows only from strong uniform mixing, like $\alpha$-mixing (see \cite{B05} for definition), and even then only at certain subsequences, which means most of the time the final result holds only for cylinders, $D^p(u_n)$ follows from decay of correlations which is much weaker and allows to obtain the result for balls, instead.

Just to give an idea of how simple it is to check $D^p(u_n)$ for systems with sufficiently fast decay of correlations, assume that for all
$\phi,\psi:M\rightarrow\R$ with bounded variation, there are $C,\alpha>0$ independent of $\phi, \psi$ and $n$ such that
\begin{equation}
\label{eq:decay-correlations} \left| \int\phi\cdot(\psi\circ
f^t)d\mu-\int\phi d\mu\int\psi d\mu\right|\leq
C\mbox{Var}(\phi)\|\psi\|_\infty \varrho(t),\quad\forall
t\in\N_0,\end{equation} where $\mbox{Var}(\phi)$ denotes the total
variation of $\phi$ (see Section~\ref{ssec:Rychlik} for more details) and $n\varrho(t_n)\to0$, as $n\to\infty$ for some $t_n=o(n)$.
Take $\phi=\I_{Q_p(u_n)}$,
$\psi=\I_{\QQ_{p,t,\ell}(u_n)}$, let $C'>0$ be such that $\mbox{Var}(\I_{Q_p(u_n)})\leq C'$, for all $n\in\N$ and set $c=CC'$. Then \eqref{eq:decay-correlations}
implies that Condition~$D^p(u_n)$ holds with
$\gamma(n,t)=\gamma(t):=c\varrho(t)$ and for the sequence
$t_n$ such that $n\varrho(t_n)\to0$, as $n\to\infty$. Observe that the existence of such $C'>0$ derives from the fact that $Q_p(u_n)$ depends only on $X_0$ and $X_p$. This is why we cannot apply the same argument to prove  $D(u_n)$ directly from Leadbetter, since we would have to take $\phi$ to be the indicator function over an event depending on an arbitrarily large number of r.v. $X_0, X_1, \ldots $, which could imply the variation to be unbounded.

 It could happen that decay of correlations is only available for H\"older continuous functions against $L^\infty$ ones, instead. This means that we cannot use immediately the test function $\phi=\I_{Q_p(u_n)}$, as we did before. However, proceeding as in \cite[Lemma~3.3]{C01} or \cite[Lemma~6.1]{FFT10}, if we use a suitable H\"older approximation one can still prove $D^p(u_n)$.

Rates of decay of correlations are nowadays well known for many chaotic systems. Examples of these include hyperbolic or uniformly expanding systems as well as the non hyperbolic or non-uniformly expanding admitting, for example, inducing schemes with a well behaved return time function. In fact, in two remarkable papers Lai-Sang Young showed that the rates of decay of correlations of the original system are intimately connected with the recurrence rates of the respective induced map. This means that, basically, for all the above mentioned systems $D^p(u_n)$ can easily be checked.

In short, proving the existence of EVL or HTS/RTS with an extremal index around repelling periodic points for systems with sufficiently fast decay of correlations is reduced to proving $D_p'(u_n)$. Usually, this requires a more closed analysis of the dynamics around the periodic points. Below, we will show that $D_p'(u_n)$ holds for Rychlik maps and for the full quadratic map, which are chaotic systems with exponential decay of correlations, and in this way obtain the existence of an extremal index different to 1. Up to our knowledge, these are the first limiting laws of rare event with an extremal index different to 1 to be proven for balls rather than cylinders.

\section{Examples of dynamical systems and observables with extremal index in $(0,1)$}
\label{sec:dyn egs}

We begin this section by introducing a particularly well-behaved class of interval maps and measures for which $\mpp(u_n)$, $D_p(u_n)$ and $D'_p(u_n)$ hold at periodic points.  This class is more general than the class of piecewise smooth uniformly hyperbolic interval maps.  We then go on to consider a particular example of a non-uniformly hyperbolic dynamical system - the `full quadratic map'.

\subsection{Rychlik systems}
\label{ssec:Rychlik}

We will introduce a class of dynamical systems considered by Rychlik in \cite{R83}.  This class includes, for example, piecewise $C^2$ uniformly expanding maps of the unit interval with the relevant physical measures.  We first need some definitions.

\begin{definition}
Given a potential $\psi:Y\to \R$ on an interval $Y$, the \emph{variation} of $\psi$ is defined as
$${\rm Var}(\psi):=\sup\left\{\sum_{i=0}^{n-1} |\psi(x_{i+1})-\psi(x_i)|\right\},$$
where the supremum is taken over all finite ordered sequences $(x_i)_{i=0}^n\subset Y$.
\end{definition}

We use the norm $\|\psi\|_{BV}= \sup|\psi|+{\rm Var}(\psi)$, which makes $BV:=\left\{\psi:Y\to \R:\|\psi\|_{BV}<\infty\right\}$ into a Banach space.

\begin{definition}[Rychlik system]\label{def:Rychlik}
$(Y,f, \phi)$ is a \emph{Rychlik system} if $Y$ is an interval, $\{Y_i\}_i$ is an at most countable collection of open intervals such that $\cup_i Y_i$ is dense in $Y$,  $f:\cup_{i} Y_i \to Y$  is a function continuous on each $Y_i$, and $\phi:Y\to [-\infty, \infty)$ is a potential such that
\begin{enumerate}
\item
$f|_{Y_i}:Y_i \to f(Y_i)$ is a diffeomorphism;
\item
${\rm Var}\ e^\phi<+\infty$, $\phi=-\infty$ on $Y \setminus \cup_i Y_i$ and $P(\phi)=0$;
\item there is a $\phi$-conformal measure $m_\phi$ on $Y$;
\item
$(f,\phi)$ is expanding: $\displaystyle \sup_{x\in Y} \phi(x) < 0$.
\end{enumerate}
\end{definition}

\begin{proposition}
Suppose that $(Y,f,\phi)$ is a topologically mixing Rychlik system, $\phi$ is H\"older continuous on each $Z_i$, and $\mu$ is the corresponding equilibrium state.
Suppose that $\zeta$ is a repelling periodic point of prime period $p$, with $\theta=\theta(\zeta)=1-e^{S_p\phi(\zeta)}\in (0,1)$.  Let $(u_n)_{n\in\N}$ be such that $n\mu(X_0>u_n)=n(1-F(u_n))\to\tau$,
as $n\to\infty$, for some $\tau\geq 0$.  Then  $D^p(u_n)$ and $D_p'(u_n)$ hold for the stochastic process $X_0, X_1, X_2,\ldots$ defined by \eqref{eq:def-stat-stoch-proc-DS}, with $\varphi$ given by \eqref{eq:observable-form-eq-st}, and we have an Extreme Value Law with extremal index $\theta$.
\label{prop:D rho n}
\end{proposition}

The idea behind the proof is that it takes a point in $Q_p(u_n)$ at least something of the order $\log n$ iterations to return to $Q_p(u_n)$.  Then after $\log n$ iterates, the decay of correlations estimates take over to give $D'_p(u_n)$.

We first give a theorem and a lemma.

\begin{theorem}[\cite{R83}]
Suppose that $(Y,f,\phi)$ is a topologically mixing Rychlik system.  Then there exists an equilibrium state $\mu_\phi=h m_\phi$ where $h\in BV$ is strictly positive and $m_\phi$ and $\mu_\phi$ are non-atomic and $(Y, f, \mu_\phi)$ has exponential decay of correlations, i.e., there exist $C>0$ and $\gamma\in (0,1)$ such that
$$\left|\int\psi\circ f^n\cdot \upsilon~d\mu_\phi- \int\psi~d\mu_\phi \int\upsilon~d\mu_\phi\right| \le C\|\psi\|_{L^1(m_\phi)}\|\upsilon\|_{BV} \gamma^n,$$
for any $\psi\in L^1(m_\phi)$ and $\upsilon\in BV$.
\label{thm:rych decay}
\end{theorem}

The fact that these maps have decay of correlations of observables in a strong norm like BV against $L^1$ observables allows us to prove the following lemma which is very similar to the first computations in the proof of \cite[Theorem 3.2]{BSTV03}.
\begin{lemma}
There exists $C'>0$ such that for all $j\in\N$
\begin{align*}
\mu_\phi\left(Q_p(u_n)\cap f^{-j}(Q_p(u_n))\right) \le
\mu_\phi(Q_p(u_n))\left(C'e^{-\beta j} + \mu_\phi(Q_p(u_n))\right).
\end{align*}
\label{lem:BSTV}
\end{lemma}

\begin{proof}
Taking $\psi=\upsilon=\I_{Q_p(u_n)}$ in Theorem~\ref{thm:rych decay} we easily get
\begin{align*}
\mu_\phi\left(Q_p(u_n)\cap f^{-j}(Q_p(u_n))\right) \le
 \mu_\phi(Q_p(u_n))^2+C \left\| \I_{Q_p(u_n)}\right\|_{BV}m_\phi(Q_p(u_n)) e^{-\beta j}
\end{align*}
Since we have assumed, as above, that $\frac{d\mu_\phi}{dm_\phi}\in BV$ and is strictly positive, and since  $\left\| 1_{Q_p(u_n)} \right\|_{BV} \le 5$ there is $C'>0$ as required.
\end{proof}

\begin{proof}[Proof of Proposition~\ref{prop:D rho n}]
First observe that the non-atomicity of $\mu$, given by Theorem~\ref{thm:rych decay}, implies that we can indeed find a suitable sequence $(u_n)_n$ as in \eqref{eq:un}, see the discussion around \eqref{eq:def-h}.
Moreover, condition $D^p(u_n)$ follows from Theorem~\ref{thm:rych decay} as in Section~\ref{subsubsec:roles-Dp-D'p}.

To prove $D_p'(u_n)$, first let $\overline U\ni \zeta$ denote a domain such that $x\in U$ implies $d(f^p(x), \zeta)>d(x, \zeta)$.
In order for a point in $Q_p(u_n)$ to return to $Q_p(u_n)$ at time $k\in \N$, there must be some time $\ell\le k/p$ such that image $f^{\ell p}(Q_p(u_n))$ must have only just escaped from the domain $U$.  Therefore we must have $\mu(f^{(\ell-1)p}(Q_p^*(u_n)))\ge C\mu(U)$ for some $C>0$ which depends only on $U$ and $\zeta$.  Since $\mu(Q_p(u_n))\sim \frac{\tau\theta}n$, $e^{S_p\phi(\zeta)}\in (0,1)$ and $$m_\phi(f^{(\ell-1) p}(Q_p(u_n)))= \int_{Q_p(u_n)}e^{S_{(\ell-1) p}\phi}~dm_\phi,$$
we must have $\ell$, and therefore $k$, greater than $B\log n$ for some $B>0$, depending on $C, U$ and $\frac{d\mu}{dm}$.

Using this and Lemma~\ref{lem:BSTV},
\begin{align*}
n\sum_{j=1}^{[n/k_n]}& \p(\{X_0\in Q_p(u_n)\}\cap
\{X_j\in Q_p(u_n)\}) \approx n\sum_{j=B\log n}^{[n/k]}\p(\{X_0\in Q_p(u_n)\}\cap \{X_j\in Q_p(u_n)\})\\
&\le n\left([n/k]-B\log n\right)\mu(Q_p(u_n))^2 +n\mu(Q_p(u_n)) e^{-B\beta\log n} \sum_{j=1}^{[n/k]-B\log n}e^{-\beta j} \\
&\le \frac{(n\mu(Q_p(u_n)))^2}k+ C_\beta n\mu(Q_p(u_n)) n^{-B\beta}
\end{align*}
where $C_\beta:=\sum_{j=0}^\infty e^{-j\beta}$.  Since $n\mu(Q_p(u_n)\to \tau\theta$ as $n \to \infty$ we have for some $C>0$
$$\lim_{n\to\infty} n\sum_{j=1}^{[n/k_n]} \p(\{X_0\in Q_p(u_n)\}\cap
\{X_j\in Q_p(u_n)\}) \le \lim_{n\to\infty} C \frac{(\tau\theta)^2}{k_n}=0,$$
as required.
\end{proof}

We give a short list of some of the simplest examples of Rychlik systems:

\begin{list}{$\bullet$}
{ \itemsep 1.0mm \topsep 0.0mm \leftmargin=7mm}
\item Given $m\in \{2, 3, \ldots\}$, let $f: x\mapsto mx \mod 1$ and $\phi\equiv-\log m$.  Then $m_\phi=\mu_\phi=\l$.
\item Let $f:x\mapsto 2x \mod 1$ and for $\alpha\in (0,1)$, let \begin{equation*}\phi(x):= \begin{cases}
    -\log\alpha & \text{ if } x\in (0,1/2)\\
    -\log(1-\alpha) & \text{ if } x\in (1/2,1)
    \end{cases}
    \end{equation*}
    (and $\phi=-\infty$ elsewhere).  Then $m_\phi=\mu_\phi$ is the $(\alpha, 1-\alpha)$-Bernoulli measure on $[0,1]$.
\item Let $f:(0,1] \to (0,1]$ and $\phi:(0,1]\to (-\infty, 0)$ be defined as $f(x)=2^k(x-2^{-k})$ and $\phi(x):=-k\log2$ for $x\in (2^{-k}, 2^{-k+1}]$.  Then $m_\phi=\mu_\phi=\l$.
\end{list}

\begin{remark}
\label{rem:multidimensional}
The crucial point in proving the result for Rychlik maps is the fact that 
the exponential decay of correlations given by Theorem~\ref{thm:rych decay} is expressed in terms of the $L^1$ norm of one of the observables. This is key in proving Lemma~\ref{lem:BSTV} also. 
In particular, the same argument can be applied to a generalisation of Rychlik maps in higher dimensions which were both defined, and proved to have decay of correlations of the same type (with an $L^1$ norm estimate), in \cite{S00}.
\end{remark}

\subsection{Quadratic Chebyshev polynomial}
\label{subsubsec:quadratic}

Let $\X=[-1,1]$ be equipped with the usual metric and Lebesgue measure defined on the Borelean sets of the interval. Let $f:[-1,1]\to[-1,1]$ be given by $f(x)=1-2x^2$.  This map is known as the full quadratic map or the quadratic Chebyshev polynomial.

It is well known that the invariant density $\mu$ is given by $$\frac{d\mu}{d\l}(x)=\frac1\pi\frac1{\sqrt {(1-x)(1+x)}}.$$ We will consider the fixed point $\zeta=\-1$ and the observable $\varphi:[-1,1]\to\R$ given by $\varphi(x)=-x$ which achieves the maximum value $1$ at $\zeta=-1$. Notice that $\varphi$ can be written as $\varphi(x)=g(\dist(x,\zeta))$ as in \eqref{eq:observable-form}, simply by taking $g:[0,\infty)\to\R$ defined by $g(y)=1-y$. Clearly, since $f'(\zeta)=4>1$, then $\zeta$ is a repelling periodic point with $p=1$.  Let $F$ denote the d.f. of $X_0$. We have that for $s$ close to $0$, the tail of the d.f. may be written as:
\begin{equation*}
1-F(1-s)=\int_{-1}^{-1+s} \frac1\pi\frac1{\sqrt {(1-x)(1+x)}} dx =\frac12+\frac1\pi \arcsin(-1+s)\sim \frac{\sqrt {2}}{\pi}\sqrt{s}.
\end{equation*}
This implies that the level $u_n$, which is such that $n(1-F(u_n))\to\tau\geq0$, as $n\to \infty$, may be written as
\begin{equation}
\label{eq:un-quad}
u_n\sim 1-\frac{(\pi\tau)^2}{2}\frac{1}{n^2}.
\end{equation}

This system has exponential decay of correlations which, in light of the discussion in Subsection~\ref{subsubsec:roles-Dp-D'p}, implies that Condition~$D^p(u_n)$ holds. In fact,
from \cite{KN92,Y92} one has that for all
$\phi,\psi:M\rightarrow\R$ with bounded variation, there is
$C,\alpha>0$ independent of $\phi, \psi$ and $t$ such that
\begin{equation}
\label{eq:decay-correlations-quadratic} \left| \int\phi\cdot(\psi\circ
f^t)~d\mu-\int\phi ~d\mu\int\psi ~d\mu\right|\leq
C\mbox{Var}(\phi)\|\psi\|_\infty \e^{-\alpha t},\quad\forall
t\in\N_0,\end{equation} where $\mbox{Var}(\phi)$ denotes the total
variation of $\phi$.
In particular, taking $\upsilon=\I_{Q_{p,0}(u_n)}$ and
$\psi=\I_{\QQ_{p,t,\ell}(u_n)}$, then \eqref{eq:decay-correlations-quadratic}
implies that Condition~$D^p(u_n)$ holds with
$\gamma(n,t)=\gamma(t):=2C\e^{-\alpha t}$ and for the sequence
$t_n=\sqrt n$, for example.

To check Condition~$D'_p(u_n)$ we have to look at the particular behaviour of the system and estimate the probability of starting in a neighbourhood of $\zeta$ and returning in relatively few iterates. The idea is to observe that since the critical orbit ends in $\zeta$, one can just estimate the probability of starting close to the critical point and returning close to it. This can easily be done by using the estimates in \cite[Section~6]{FF08} where $D'(u_n)$ was proved for Benedicks-Carleson maps (which include the example in hand) for observables achieving a maximum either at the critical point or at its image. Note that, here, $D'(u_n)$, which imposes some significant memory loss for relatively fast returns, cannot hold because $\zeta$ is a fixed point. Nevertheless, what we will prove is that the set points that manage to escape from a tight vicinity of $\zeta$, \ie $Q_{p,0}(u_n)$, only return after having a significant memory loss.

As in \cite[Section~6]{FF08}, we start by computing a turning instant $T$ which splits the time interval $0,\ldots,[n/k]$ of the sum in \eqref{eq:D'rho-un} into two parts.
We compute $T\in\N$ such that for every $j>T$ we have $2C\e^{-\alpha j}<\frac1{n^3}.$ For $n$ large enough it suffices to take
\[
T=\left[\frac{4\log n}{\alpha}\right].
\]
From \eqref{eq:decay-correlations-quadratic} with $\upsilon=\psi=\I_{Q_p(u_n)}$ and for $j>T$ one easily gets
\begin{equation*}
\left|\mu(Q_{p,0}(u_n)\cap Q_{p,j}(u_n))-\mu(Q_{p,0}(u_n))^2\right|\leq \frac{1}{n^3},
\end{equation*}
which implies that for some $C>0$ we have
\begin{align*}
n\sum_{j=T}^{[n/k_n]}\mu(Q_{p,0}(u_n)\cap Q_{p,j}(u_n))&\leq n[n/k_n]\left(\mu(Q_{p,0}(u_n))^2+\frac 1{n^3}\right)\\
&\leq \frac{n^2}{k_n} \mu(X_0>u_n)^2+\frac{n^2}{k_n n^3}\leq C \frac{\tau^2}{k_n}\xrightarrow[n\to\infty]{}0.
\end{align*}

Thus, we are left with the piece of history from time $0$ to $T$ to analyse. For that purpose, we start by studying the pre-images of a small interval in the vicinity of $\zeta$, namely, $A(s)=[-1,-1+s]$ for $s$ close to $0$. We have
$
f^{-1}(A(s))=A_1(s)\cup A_2(s),
$
where $A_1(s)=[-1,-\sqrt{1-s/2}]$ is a small neighbourhood of $-1$ and $A_2(s)=[\sqrt{1-s/2},1]$ is a small neighbourhood of $1$. Also, the pre-image of any small neighbourhood of $1$ is a small neighbourhood of the critical point $0$, in particular:  $$f^{-1}(A_2(s))=\left[-\sqrt{1/2-1/2\sqrt{1-s/2}},\sqrt{1/2-1/2\sqrt{1-s/2}}\right].$$ Moreover, for $s$ close to $0$, we may write
\begin{equation}
\label{eq:quad-aux1}
\sqrt{\frac12-\frac12\sqrt{1-\frac s2}}\sim \frac{\sqrt2}4\sqrt s.
\end{equation}
Now, observe that $Q_p(u_n)\subset A(1-u_n)$ and if you are to enter $A(1-u_n)$ then either you are in $A_1(1-u_n)\subset A(1-u_n)$ or in $A_2(1-u_n)$. Since, by definition of $Q_p(u_n)$, if you start in a point of $Q_p(u_n)$, then you immediately leave $A(1-u_n)$ in the next iterate, this means that the only way you can return to  $Q_p(u_n)$ is if you enter $A_2(1-u_n)$, which implies that you must enter the neighbourhood of the critical point $f^{-1}(A_2(1-u_n))$ first. Hence, $$Q_p(u_n)\cap f^{-j}(Q_p(u_n))\subset A(1-u_n)\cap f^{-j+2}(f^{-1}(A_2(1-u_n)).$$
Using the symmetry of the map and of the invariant density plus the invariance of $\mu$, we have
\begin{align*}
\mu\left(A(1-u_n)\cap f^{-j+2}(f^{-1}(A_2(1-u_n))\right)&=2\mu\left(A_2(1-u_n)\cap f^{-j+1}(f^{-1}(A_2(1-u_n))\right)\\
&=2\mu\left(f^{-1}(A_2(1-u_n))\cap f^{-j}(f^{-1}(A_2(1-u_n))\right).\end{align*}
This means, that we only need to study the probability of starting in the neighbourhood of the critical point and returning after $j$ iterates and for that we use the computations in \cite[Section~6]{FF08}. The threshold $\Theta$, defined in \cite[Equation (6.1)]{FF08},  here is given by
\begin{equation*}
\Theta=\Theta(n)=\left[-\log\left(\sqrt{1/2-1/2\sqrt{1-(1-u_n)/2}}\right)\right].
\end{equation*}
Using \eqref{eq:un-quad} and \eqref{eq:quad-aux1} we may write
\[
\Theta(n)\sim -\log\left(\frac{\sqrt2}4\sqrt{1-u_n}\right)\sim\log n,
\]
which implies the existence of $C_1>0$ such that $2T/\Theta\leq C_1$ for all $n\in\N$. We can now use the final computations of \cite[Section~6]{FF08} to get that
\[
\mu\left(f^{-1}(A_2(1-u_n))\cap f^{-j}(f^{-1}(A_2(1-u_n))\right)\leq \mbox{const} \,T^{C_1+1}\e^{-(1-7\beta)2\Theta},
\]
where $0<\beta<0.01$. Finally, since by definition of $\Theta$, \eqref{eq:un-quad} and \eqref{eq:quad-aux1}, we have $\e^{-\Theta}\leq \mbox{const}\, 1/n$, it follows that
\begin{align*}
n\sum_{j=1}^{T}\mu(Q_p(u_n)\cap f^{-j}(Q_p(u_n)))&\leq \mbox{const}\, n\sum_{j=1}^{T}T^{C_1+1}\e^{-2(1-7\beta)\Theta}\\
&\leq \mbox{const}\, \frac{n\,(\log n)^{C_1+2}}{n^{2(1-7\beta)}}\xrightarrow[\n\to\infty]{}0
\end{align*}
This means that $D'_p(u_n)$ also holds and, hence, by Theorem~\ref{thm:EI-abs-cont}, we have an EVL with extremal index $\theta=1/2$ for the stochastic process defined by \eqref{eq:def-stat-stoch-proc-DS} with $\varphi$ defined above and achieving a global maximum at the repelling fixed point $\zeta=-1$.

\section{EVL/HTS for cylinders}
\label{sec:cyl}

Many results on HTS for dynamical systems were initially proved for HTS to dynamically defined cylinders, which is usually a more straightforward problem to study.  Indeed many results which are known about the statistics of hits to cylinders are not known for balls.  Therefore one of our goals in \cite{FFT11} was to extend the results of \cite{FFT10} to this setting.  In this short section we outline this theory and in Section~\ref{sec:ex ind imples per} we will apply it to the problem of EVLs with non-trivial EI.

Many dynamical systems $(\X,f)$ come with a natural partition $\P_1$, for example this might be the collection of maximal sets on which $f$ is locally homeomorphic.  The dynamically defined $n$-cylinders are then $\P_n:=\bigvee_{i=0}^{n-1}f^{-i}(\P_1)$.  For $x\in \X$, let $\cyl_n[x]$ denote an element of $\P_n$ containing $x$.  Note that in principle there may be more than one choice of cylinder, but in the cases we consider we can make an arbitrary choice.

If we wish to deal with HTS/EVL to dynamically defined cylinders $\cyl_n[\zeta]$ around a point $\zeta$, we replace the sets $(U_n)_n$ with $(\cyl_n[\zeta])_n$ in \eqref{eq:def-HTS-law}. In this case we chose our observable $\varphi$ to be of the form
\begin{equation}
\label{eq:def-observable-cylinders}
\varphi=g_i\circ \psi,
\end{equation}
where $g_i$ is one of the three forms given above and $\psi(x):= \mu(\cyl_n[\zeta])$ where $n$ is maximal such that $x\in \cyl_n[\zeta]$.
Moreover, we select  a subsequence of the time $n$, which we denote by $(\omega_n)_{n\in\N}$ and such that
\begin{equation}
  \label{eq:wn-condition}
  \omega_n \mu(X_0>u_n)\xrightarrow[n\to\infty]{}\tau>0,
  \end{equation}
and for every $n\in\N$, $\tau\geq 0$, where $u_n$ is taken to be such that
\begin{equation}
\label{eq:un-cylinders}
\{X_0>u_n\}=\cyl_n[\zeta].
\end{equation} We achieve this, for example, by letting
\begin{equation}
\label{eq:wn-definition}
\omega_n=\omega_n(\tau)=[\tau \left(\mu(X_0>u_n)\right)^{-1}].
\end{equation}

Finally, we say that we have a \emph{cylinder EVL} $H$ for the maximum if  for any sequence $(u_n)_{n\in\N}$ such that \eqref{eq:un-cylinders} holds and for $\omega_n$ defined in \eqref{eq:wn-definition}, the limit \eqref{eq:wn-condition} holds and
\begin{equation}
\mu\left(M_{\omega_n}\le u_n\right)\to \bar H(\tau),
\label{eq:cyl law}
\end{equation}
for some non-degenerate d.f. $H$, as
$n\to\infty$.  The cylinder HTS is defined analogously.  The equivalence between these two perspectives was given in \cite[Theorem 3]{FFT11}. We also showed in that paper that the following two conditions imply that \eqref{eq:cyl law} holds with $\bar H(\tau)=e^{-\tau}$.
\begin{condition}[$D(u_n,\omega_n)$]\label{cond:D-cyl}We say that $D(u_n,\omega_n)$ holds
for the sequence $X_0,X_1,\ldots$ if for any integers $\ell,t$
and $n$
\[ \left|\mu\left(\{X_0>u_n\}\cap
  \{\max\{X_{t},\ldots,X_{t+\ell-1}\}\leq u_n\}\right)-\mu(\{X_0>u_n\})
  \mu(\{M_{\ell}\leq u_n\})\right|\leq \gamma(n,t),
\]
where $\gamma(n,t)$ is nonincreasing in $t$ for each $n$ and
$\omega_n\gamma(n,t_n)\to0$ as $n\rightarrow\infty$ for some sequence
$t_n=o(\omega_n)$.
\end{condition}
\begin{condition}[$D'(u_n,\omega_n)$]\label{cond:D'un-cyl} We say that $D'(u_n,\omega_n)$
holds for the sequence $X_0,X_1,\ldots$ if
\begin{equation}
\label{eq:D'un-cyl}
\lim_{k\rightarrow\infty}
\limsup_{n\rightarrow\infty}\,\omega_n\sum_{j=1}^{\lfloor \omega_n/k\rfloor}
\mu(\{X_0>u_n\}\cap
\{X_j>u_n\})=0.
\end{equation}
\end{condition}

\begin{remark}
We say a system is \emph{$\Phi$-mixing}, if for an $n$-cylinder $U$ and a measurable set $V$,
$$\left|\mu(U\cap f^{-j}(V))-\mu(U)\mu(V)\right| \le \Phi(j)\mu(U)\mu(V)$$ where $\Phi(j)$ decreases to 0 monotonically in $j$. This holds in the Axiom A case: see Haydn and Vaienti \cite{HV09} for example. In \cite[Section 3]{HV09} they showed that $\Phi$-mixing dynamical systems give rise to Poisson HTS around periodic points with a parameter, interpreted in the current paper as the EI.  The Poisson law is for the number of returns to asymptotically small cylinders.  Our results imply theirs for the \emph{first} hitting time.  As can be seen from our examples, we do not require our systems to have such good mixing properties and moreover our results also apply to balls.
\end{remark}

\section{Dichotomy for uniformly expanding maps}
\label{sec:ex ind imples per}

In this section we will prove that for a simple class of dynamical systems periodic points are the only points which can generate a cylinder EVL with EI in $(0,1)$.  Therefore we understand all the cylinder EVLs for this system.

We assume that the dynamics is $f:x\mapsto 2x \mod 1$ on the unit interval $I=[0,1]$. Let $\alpha\in (0,1/2]$ and $\mu$ be the $(\alpha, 1-\alpha)$-Bernoulli measure.  This is thus a Rychlik system as in Section~\ref{ssec:Rychlik}.  Moreover, for $\alpha=1/2$ the measure is Lebesgue. While this system has stronger mixing properties, we will only actually use the fact that for our system, for $n$-cylinders $U, V$,
\begin{equation}
\left|\mu(U\cap f^{-j}(V))-\mu(U)\mu(V)\right| \le \Phi(j)\mu(U) \text{ where } \sum_j\Phi(j)<\infty.
\label{eq:phish mix}
\end{equation}
(Observe that this is a weaker assumption than $\Phi$-mixing.)

\begin{proposition}
Suppose that $(I, f, \mu)$ and the observable $\varphi$ is as in \eqref{eq:def-observable-cylinders}.  If $\zeta\in I$ is non-periodic then $D'(u_n,\omega_n)$ and $D(u_n,\omega_n)$ hold.  Hence there is an EVL with EI equal to 1.
\label{prop:cyl law}
\end{proposition}

\begin{remark}
\begin{list}{$\bullet$}
{ \itemsep 1.0mm \topsep 0.0mm \leftmargin=7mm}

\item Given a periodic point $\zeta$ of prime period $q$, if another periodic point $x\neq \zeta$, with prime period $p>q$, shadows the orbit of $\zeta$ for a long time, say for $n<p$ steps, but differs at some stage from the orbit of $\zeta$, then the EI corresponding to the point $x$ is of the form $1-\alpha^k(1-\alpha)^{p-k}$ for some $0\le k\le p$.  So if $p$ is very large then the EI here is almost 1.  This implies that the EI corresponding to $x$ has no relationship with the EI corresponding to $\zeta$, no matter how much shadowing takes place.  Arguing heuristically that non-periodic points should behave like periodic points with very long period, this idea also suggests that our dichotomy should hold for a much larger class of dynamical systems.

\item This proposition allied to the cylinder version of Proposition~\ref{prop:D rho n}, completely characterises the possible cylinder EVLs for this system.
\item In the proof of the proposition, the only properties we need for our dynamical system are that it is Markov, that \eqref{eq:phish mix} holds and the measure of $n$-cylinders decay exponentially in $n$.

\item We would expect a similar proposition to be true for balls also.  However we strongly use the cylinder structure of the system $(I, f)$ in our proof.  It may be possible to approximate the balls by cylinders, but and since $n$-cylinders can not all be assumed to be symmetric about $\zeta$, in the usual metric on $I$,  this may not be straightforward.
\end{list}
\end{remark}

Before proving the proposition, we will discuss the symbolic structure of our dynamical system and then prove a lemma.

First we recall that the system $(I,f)$ has a natural coding:  $x\in I$ can be given the code $x_0x_1\ldots$ where $x_i=0$ if $f^i(x)\in [0,1/2)$ and $x_i=1$ if $f^i(x)\in [1/2,1)$. Then the dynamics is semi-conjugate to the full shift on two symbols $(\{0,1\}^{\N_0}, \sigma)$, where $\sigma(x_0 x_1\ldots) =x_1 x_2 \ldots$ for $x_i\in \{0,1\}$. Notice that the points $x$ where there is a problem in the conjugacy are precisely the points which map onto the fixed point at zero. Following the proofs below it is easy to see that Proposition~\ref{prop:cyl law} follows almost immediately in this case.  In fact this is also the situation in which the cylinder $\cyl_n[x]$ is not well defined.

Now let $(p_i)_i$ be the sequence of integers such that whenever $p_i\le n<p_{i+1}$ the time the orbit of $\zeta$ takes to visit $\cyl_n[\zeta]$ is at least $p_i$. (We will sometimes denote $i$ such that $p_{i_n}\le n<p_{i_n+1}$ by $i_n$.)  For example, suppose that the first 153 symbols representing $\zeta$ are
\begin{align} & 000000000000001\ 000000000000001\ 000000000000001\ 000000000000001 \notag\\
 & 000000000000001\ 000000000000001\ 000000000000001\ 000000000000001 \label{eq:symb eg}\\ & 000000000000001\ 000000000000001\ 001.\notag\end{align}
In this case $p_0=1$, $p_1=15$ and $p_2=153$.

Letting $a_n\in \N$ be maximal such that $a_np_i\le n$, we can also interpret $(p_i)_i$ as the first times $r=p_i$ when $\cyl_r[\zeta]$ contains no periodic point of period less than $r$.  So for $p_i\le n<p_{i+1}$ where $j=a_np_i +q_n$, for some $0\le q_n<p_i$, the coding for $\zeta$ up to time $n$ must consist of the block $\zeta_0\ldots \zeta_{p_i-1}$ repeated $a_n$ times followed by the block $\zeta_0\ldots \zeta_{q_n-1}$.

\begin{remark}
The periodic structure of cylinders was considered in \cite{AVe09}, see particularly Section 3 (note that there they are interested in first returns/hitting times of the \emph{whole cylinder} to itself, which is slightly different to what we look at here).  They considered the $p_i$ blocks $\zeta_0\ldots \zeta_{p_i-1}$ as being `$i$-period' blocks and the block $\zeta_{a_np_i}\ldots \zeta_{a_np_i+q_n}$ as an `$i$-rest' block.  In the example in \eqref{eq:symb eg} the relevant blocks have 1-period 15 and the 1-rest period is 2.
\end{remark}

 A key point in the proof of Proposition~\ref{prop:cyl law} is that the assumption that $\zeta$ is not periodic implies that $p_{i_n}\to \infty$ as $n\to \infty$. The following lemma explains how this affects short term returns to $n$-cylinders.

\begin{lemma}
For $p_i\le n<p_{i+1}$ as above, if, for $j\le n$, there is a cylinder $\cyl_{n+j}\subset \cyl_n[\zeta]$ such that $f^j(\cyl_{n+j})\subset \cyl_n[\zeta]$ then
\begin{enumerate}[(a)]
\item there is only one such cylinder in $\cyl_n[\zeta]$ with the same return time $j$;
\item there exists $0\le k\le a_n$ such that $j=kp_i$.
\end{enumerate}
\label{lem:sparse early returns}
\end{lemma}

\begin{proof}
Suppose that $\cyl_{n+j}$ is as in the lemma.  Since $\cyl_{n+j}\subset \cyl_n[\zeta]$, the coding for $\cyl_{n+j}$ must be of the form $$\zeta_0\ldots\zeta_{n-1}\alpha_0\ldots\alpha_{j-1}$$ for some $\alpha_0\ldots\alpha_{j-1}\in \{0,1\}^j$.  Moreover the fact that $j\le n$ means that the coding for $\cyl_{n+j}$ must be $$\zeta_0\ldots \zeta_{j-1}\zeta_{0} \ldots \zeta_{n-1}.$$
In particular there is only one possible code for such a cylinder, determined only by $j$ and by the first $n$ entries in the code for $\zeta$ and the first part of the lemma follows.  The second part also follows immediately from the periodic structure of the code for $\zeta$.   (The proof can also be seen from the setup described in \cite[Section 3]{AVe09}.)
\end{proof}

\begin{proof}[Proof of Proposition~\ref{prop:cyl law}]
The fact that $D(u_n,\omega_n)$ holds follows from Theorem~\ref{thm:rych decay} as in Section~\ref{subsubsec:roles-Dp-D'p}.

To prove $D'(u_n,\omega_n)$, we first estimate the first $n$ terms in the sum \eqref{eq:D'un-cyl}.  We note that for our system and $\cyl_{n+j}$ as in Lemma~\ref{lem:sparse early returns}
 $$\mu(\cyl_{n+j})\le \mu(\cyl_n[\zeta])\vartheta^j$$ for $\vartheta=1/\alpha$.
Then for $a_n$ maximal such that $a_np_{i_n}\le n$ and $\omega_n=\omega_n(\tau)$
\begin{align*}
\omega_n\sum_{j=1}^{n}
\mu\left(x\in \cyl_n[\zeta]:f^j(x)\in \cyl_n[\zeta]\right) &= \omega_n\sum_{j=p_{i_n}}^{n} \mu\left(x\in \cyl_n[\zeta]:f^{j}(x)\in \cyl_n[\zeta]\right)\\
&=\omega_n\sum_{k=1}^{a_n} \mu\left(x\in \cyl_n[\zeta]:f^{kp_i}(x)\in \cyl_n[\zeta]\right)\\
&\lesssim \omega_n\mu(\cyl_n[\zeta])\sum_{k=1}^{\infty} \vartheta^{kp_i} \lesssim \frac{\tau \vartheta^{p_i}}{1-\vartheta}.
\end{align*}
Therefore using the mixing condition,
\begin{align*}
\omega_n\sum_{j=1}^{\lfloor n/k\rfloor}
\mu(\{X_0>u_n\}\cap \{X_j>u_n\})& =\omega_n\sum_{j=1}^{n}
\mu\left(x\in \cyl_n[\zeta]:f^j(x)\in \cyl_n[\zeta]\right)\\
&\quad +\omega_n\sum_{j=n+1}^{\lfloor \omega_n/k\rfloor}
\mu(x\in \cyl_n[\zeta]:f^j(x)\in \cyl_n[\zeta])\\
&\lesssim \frac{\tau \vartheta^{p_i}}{1-\vartheta}+ \omega_n\sum_{j=n+1}^{\lfloor \omega_n/k\rfloor}\Phi(j)\mu(\cyl_n[\zeta])+ \mu(\cyl_n[\zeta])^2\\
&\le \tau\left(\frac{\vartheta^{p_i}}{1-\vartheta}+\frac\tau k+\sum_{j=n+1}^{\infty}\Phi(j)\right).
\end{align*}

Since $\Phi(j)$ is summable,
$$\sum_{j=n+1}^{\infty}\Phi(j) \to 0 \text{ as } n\to \infty.$$  Moreover, as $n\to \infty$, our assumption that $\zeta$ is not periodic implies that $p_{i_n}\to\infty$ as $n\to \infty$.  Therefore,
$$\limsup_{n\to \infty}\omega_n\sum_{j=1}^{\lfloor \omega_n/k\rfloor}
\mu(\{X_0>u_n\}\cap \{X_j>u_n\})\le \frac{\tau^2}k,$$
$D'(u_n,\omega_n)$ follows by taking $k\to \infty$.  The existence of an EVL with EI equal to 1 follows from \cite[Section 5]{FFT11}.
\end{proof}

\section*{Acknowledgements.}
We would like to thank Sandro Vaienti for encouragement and fruitful conversations on this subject.

\appendix

\section{A Maximum Moving Average process with period $2$}
\label{sec:MMA-101}

In all sections of the appendix we show how our theory applies to some classical stochastic processes.

Let $Y_{-2},Y_{-1},Y_0,Y_1,\ldots$ be a sequence of i.i.d. random variables with common d.f. $G$. We require that for all $\tau\geq0$ there exists a sequence $\{v_n\}_{n\in\N}$ such that $n(1-G(v_n))\to\tau$, as $n\to\infty$. We define a Maximum Moving Average process $X_0,X_1,\ldots$ based on the previous sequence in the following way: for each $n\in\N_0$ set
\begin{equation*}
X_n=\max\{Y_{n-2},Y_n\}.
\end{equation*}
Note that by definition of the sequence $\{u_n\}_{n\in\N}$, we must have $\p(X_0>u_n)\to\tau\geq 0$, as $n\to\infty$. For simplicity let $\alpha_n=\p(Y_0\leq u_n)$. Then, since $\p(X_0>u_n)=2(1-\alpha_n)-(1-\alpha_n)^2$ we must have that $\alpha_n\to 1$ and $n(1-\alpha_n)\to\tau/2\geq 0$, as $n\to\infty$.

We claim that the $2$-dependent process $X_0,X_1,\ldots$ satisfies $\spp$, with $p=2$, $\theta=1/2$, $D^p(u_n)$ and $D'_p(u_n)$. Applying Theorem~\ref{thm:existence-EI} it follows that $X_0,X_1,\ldots$ has an EI given by $\theta=1/2$.

We start by verifying $\spp$.  Since $\p(X_0>u_n)=2(1-\alpha_n)-(1-\alpha_n)^2$ and $\p(X_1\leq u_n, X_0>u_n)=2\alpha_n^2(1-\alpha_n)-\alpha_n^2(1-\alpha_n)^2$, it follows
\[
\p(X_1>u_n|X_0>u_n)=1-\frac{2\alpha_n^2(1-\alpha_n)-\alpha_n^2(1-\alpha_n)^2}{2(1-\alpha_n)-(1-\alpha_n)^2}\xrightarrow[n\to\infty]{}0.
\]
On the other hand since $\p(X_2\leq u_n,X_0>u_n)=\alpha_n^2(1-\alpha_n)$ we have
\[
\p(X_2>u_n|X_0>u_n)=1-\frac{\alpha_n^2(1-\alpha_n)}{2(1-\alpha_n)-(1-\alpha_n)^2}\xrightarrow[n\to\infty]{}\frac12,
\]
which means that \eqref{cond:periodicity} is satisfied with $p=2$ and $\theta=1/2$. Condition \eqref{cond:summability} follows from the fact that $X_0,X_1,\ldots$ is $2$-dependent. In fact, since for all $i\in\N$
\begin{multline*}
\p\left(X_0>u_n, X_2>u_n, \dots, X_{2(2i+1)}>u_n\right)\leq \p\left(X_0>u_n, X_2>u_n, \dots, X_{2(2i)}>u_n\right)\\
\leq\p\left(X_0>u_n, X_4>u_n, \dots, X_{4i}>u_n\right)\leq 2^i(1-\alpha_n)^i
\end{multline*}
and $(1-\alpha_n)\to0$, it follows that there exists $C>0$ such that
\begin{equation*}
\sum_{i=1}^n \p\left(X_0>u_n, X_2>u_n, \dots, X_{2i}\right) \leq 2\sum_{i=1}^{[n/2]}2^i(1-\alpha_n)^i\leq C (1-\alpha_n)\sum_{i=1}^{\infty}(1/2)^i\xrightarrow[n\to\infty]{}0.
\end{equation*}

Condition $D^p(u_n)$ follows trivially from the fact that the process is $2$-dependent.

We are left now with condition $D'_p(u_n)$. Recall that, in this case, $Q_{p,i}(u_n)=\{X_i>u_n,X_{i+2}>u_n\}$. It is easy to check that $\p(Q_{p,0}\cap Q_{p,i})=(1-\alpha_n)^2\alpha_n^4$ for all $i\in\N$, except for $i=2$ and $i=4$ for which such probability is $0$. Hence,
\[
\sum_{i=1}^{[n/k_n]} n\p(Q_{p,0}\cap Q_{p,i})\leq [n/k_n] n(1-\alpha_n)^2\alpha_n^4 \xrightarrow[n\to\infty]{}0,
\]
because $[n/k_n] (1-\alpha_n)\to0$, $\alpha_n^4\to1$ and $n(1-\alpha_n)\to\tau/2\geq 0$, as $n\to\infty$.

\section{A Maximum Moving Average process with two underlying periodic phenomena of periods $1$ and $3$}
\label{sec:MMA-1101}
As before, let $Y_{-2},Y_{-1},Y_0,Y_1,\ldots$ be a sequence of i.i.d. random variables as in Appendix~\ref{sec:MMA-101}.  This time, we define a Maximum Moving Average process $X_0,X_1,\ldots$ in the following way: for each $n\in\N_0$ set
\begin{equation*}
X_n=\max\{Y_{n-3},Y_{n-2},Y_n\}.
\end{equation*}
This example appears also in \cite[Section~3]{F06}.
Letting $\alpha_n=\p(Y_0\leq u_n)$, since $\p(X_0>u_n)=3(1-\alpha_n)-3(1-\alpha_n)^2+(1-\alpha_n)^3$ we must have that $\alpha_n\to 1$ and $n(1-\alpha_n)\to\tau/3\geq 0$, as $n\to\infty$.

We claim that the $4$-dependent process $X_0,X_1,\ldots$ satisfies $\sppi$, with $i=1,2$, $\mathbf p_2=(p_1,p_2)=(1,3)$, $\Theta_2=(2/3,1/2)$, $D^{\mathbf p_2}(u_n)$ and $D'_{\mathbf p_2}(u_n)$. Applying Theorem~\ref{thm:existence-EI-i} it follows that $X_0,X_1,\ldots$ has an EI given by $\theta=2/3.1/2=1/3$.

We start by verifying $\spp$ with $p_1=1$ and $\theta_1=2/3$.  Since $\p(X_0>u_n)=3(1-\alpha_n)-3(1-\alpha_n)^2+(1-\alpha_n)^3$ and $\p(X_1\leq u_n, X_0>u_n)=2\alpha_n^3(1-\alpha_n)-\alpha_n^3(1-\alpha_n)^2$, it follows
\[
\p(X_1>u_n|X_0>u_n)=1-\frac{2\alpha_n^3(1-\alpha_n)-\alpha_n^3(1-\alpha_n)^2}{3(1-\alpha_n)-3(1-\alpha_n)^2+(1-\alpha_n)^3}\xrightarrow[n\to\infty]{}1/3,
\]
which means that \eqref{cond:periodicity} is satisfied with $p_1=1$ and $\theta_1=2/3$. Condition \eqref{cond:summability} follows from the fact that $X_0,X_1,\ldots$ is $4$-dependent just as in Appendix~\ref{sec:MMA-101}. This time we cannot apply Theorem~\ref{thm:existence-EI} because $D'_{p_1}$ does not hold since
\[
\p(Q_{p_1,0}(u_n)\cap Q_{p_1,3}(u_n))=\p(Q_{\mathbf p_1,0}^{(1)}(u_n)\cap Q_{\mathbf p_1,3}^{(1)}(u_n))=(1-\alpha_n)\alpha_n^5,
\]
which means that for $\tau>0$ we have $n\p(Q_{p_1,0}(u_n)\cap Q_{p_1,3}(u_n))\to\tau/3$, which, in turn, reveals the presence of another periodic phenomenon of period $3$. In fact, since $\p(Q_{p_1,0}(u_n))=2(1-\alpha_n)\alpha_n^2-(1-\alpha_n)^2\alpha_n^2$, we have
\begin{align*}
\p\left(Q_{p_1,1}(u_n)|Q_{p_1,0}(u_n)\right)&=
0\\
\p\left(Q_{p_1,2}(u_n)|Q_{p_1,0}(u_n)\right)&=
\frac{(1-\alpha_n)^2\alpha_n^5}{2(1-\alpha_n)\alpha_n^2-(1-\alpha_n)^2\alpha_n^2}\xrightarrow[n\to\infty]{}0\\
\p\left(Q_{p_1,3}(u_n)|Q_{p_1,0}(u_n)\right)&=
\frac{(1-\alpha_n)\alpha_n^5}{2(1-\alpha_n)\alpha_n^2-(1-\alpha_n)^2\alpha_n^2}\xrightarrow[n\to\infty]{}1/2,
\end{align*}
which implies that \eqref{cond:i-periodicity} holds when $i=2$. Besides, using the $4$-dependence of the process we can show that, for $i=2$, condition \eqref{cond:i-summability} holds just as in the proof of \eqref{cond:summability} in Appendix~\ref{sec:MMA-101}.

As before, the fact that $X_0,X_1,\ldots$ is $4$-dependent clearly implies that condition $D^{\mathbf p_2}(u_n)$ holds.

We are left with $D'_{\mathbf p_2}(u_n)$. To verify it we observe that
\begin{align*}
\p\left(Q_{\mathbf p_2,0}^{(2)}(u_n)\cap Q_{\mathbf p_2,1}^{(2)}(u_n)\right)&=
0\\
\p\left(Q_{\mathbf p_2,0}^{(2)}(u_n)\cap Q_{\mathbf p_2,2}^{(2)}(u_n)\right)&=
(1-\alpha_n)^3\alpha_n^5(1+\alpha_n)\\
\p\left(Q_{\mathbf p_2,0}^{(2)}(u_n)\cap Q_{\mathbf p_2,3}^{(2)}(u_n)\right)&=
0,
\end{align*}
and for all $j\geq4$ we have
\[
\p\left(Q_{\mathbf p_2,0}^{(2)}(u_n)\cap Q_{\mathbf p_2,j}^{(2)}(u_n)\right)\leq \p\left((Y_{-3}>u_n\vee Y_0>u_n),(Y_{j-2}>u_n\vee Y_j>u_n)\right)\leq 4(1-\alpha_n)^2.
\]
All these together give
\begin{align*}
n\sum_{j=1}^{[n/k_n]}\p\left(Q_{\mathbf p_2,0}^{(2)}(u_n)\cap Q_{\mathbf p_2,j}^{(2)}(u_n)\right)\leq 2n(1-\alpha_n)^3+[n/k_n]n4(1-\alpha_n)^2\xrightarrow[n\to\infty]{}0.
\end{align*}

\section{An autoregressive process due to Chernick}
\label{sec:AR1}

In this section we consider a stationary first-order autoregressive process with uniform marginal distributions introduced by Chernick \cite{C81} which illustrates the existence of an Extremal Index different from 1.

Let $r\in\N\setminus\{1\}$ and  $\epsilon_1, \epsilon_2,\ldots$ be a sequence of i.i.d. random variables with uniform discrete distribution on $\{0,1/r,2/r,\ldots,(r-1)/r\}$, \ie $\p(\epsilon_1=k/r)=1/r$ for all $k=0,1,\ldots,r-1$. The uniform $AR(1)$ process is defined recursively as follows:
\begin{equation*}
X_n=\frac1rX_{n-1}+\epsilon_n,
\end{equation*}
where $\epsilon_n$ is independent of $X_{n-1}$ and $X_0$ is uniformly distributed in $[0,1]$. It is simple to check that $X_0,X_1,\ldots$ forms a stationary stochastic process such that each $X_n$ is uniformly distributed on $[0,1]$.

In \cite[Theorem~3.1]{C81}, Chernick shows that this process satisfies $D(u_n)$ from Leadbetter but $D'(u_n)$ fails. Besides, in \cite[Theorem~4.1]{C81}, using a direct approach, he shows that the partial maxima has a EVL of type III with an extremal index equal to $1-1/r$.

We will show that machinery we developed can be applied to this process and obtain the same result as Chernick \cite[Theorem~4.1]{C81} simply by checking the conditions of Theorem~\ref{thm:existence-EI}.

\subsection{Verification of $\mpp$}
\label{subsec:AR1-PMp}
The proof of condition $\mpp$ relies on the following property of the process:
\begin{lemma}
\label{lem:AR1-property}
For all $u>(r-1)/r$ and $n\in\N$, if $X_{n-1}>u$  then $X_{n}>u$ if and only if $\epsilon_n=(r-1)/r$.
\end{lemma}
This means that the probability of having an exceedance of any high level $u$, given that you have just had an exceedance,  is $1/r$, which makes it a periodic phenomenon of  period $p=1$ (in the sense of condition~\ref{cond:periodicity}).
\begin{proof}[Proof of Lemma~\ref{lem:AR1-property}]
Fix $u>(r-1)/r$, $n\in\N$ and assume that $X_{n-1}>u$.

First we show that if $\epsilon_n=(r-1)/r$ then $X_n>u$. To see this observe that since $r>1$ we have:
\[
1-X_n=1-\epsilon_n-\frac1rX_{n-1}=\frac1r-\frac{X_{n-1}}r<\frac1r(1-u)<1-u,
\]
which gives $X_n>u$.

Finally, if $\epsilon_n\leq(r-2)/r$, we have $X_n= X_{n-1}/r+\epsilon_n\leq 1/r+(r-2)/r= (r-1)/r<u$.
\end{proof}

Letting $u>(r-1)/r$, $i\in\N$, using Lemma~\ref{lem:AR1-property} and the facts that the $\epsilon_n$'s are a sequence of iid random variables and each $\epsilon_n$ is independent of $X_{n-1}$ we have:
\begin{align*}
\p\left(X_1>u,\ldots, X_i>u | X_0>u\right)&= \frac{\p\left(X_0>u,X_1>u,\ldots, X_i>u\right)}{\p(X_0>u)}\\
	&= \frac{\p\left(X_0>u,\epsilon_1=(r-1)/r,\ldots, \epsilon_i=(r-1)/r\right)}{\p(X_0>u)}\\
	&=\frac{\p(X_0>u)\p\left(\epsilon_1=(r-1)/r\right)\ldots \p\left(\epsilon_i=(r-1)/r\right)}{\p(X_0>u)}\\
	&=(1/r)^{i}
\end{align*}
Hence, we have that for all $i\in\N$
\[
\lim_{u\to 1}\p\left(X_1>u,\ldots, X_i>u | X_0>u\right)=(1/r)^{i},
\]
which means that $\mpp$ holds with $p=1$ and $\theta=1-1/r$.

The validity of $D^p(u_n)$ follows from a trivial adaptation of the proof of $D(u_n)$ in \cite[Theorem~3.1]{C81} and, in fact, $D^p(u_n)$ holds with $\gamma(n,t)=\frac{(1/r)^{t-1}}{1-1/r}$.

\subsection{Verification of $D'_p(u_n)$}
\label{subsec:AR1-D'p}

We start by computing a turning instant $t^*$ which splits the time interval $0,\ldots,[n/k]$ of the sum in \eqref{eq:D'rho-un} into two parts. The idea behind this splitting is that the dependence between the random variables $X_j$ in the second part, \ie with $j> t^*$ and $X_0$ is negligible because of the fast (exponential) decay of $\gamma(n,t)$ in $t$. This will leave us to analyse the time interval $0,\ldots, t^*$.

Let $\rho=1/r$. We compute $t^*\in\N$ such that for every $j>t^*$ we have $$\frac{\rho^{j-1}}{1-\rho}<\frac1{n^3}.$$ Since $\frac{\rho^{j-1}}{1-\rho}<\frac1{n^3}\Rightarrow j> \frac{3\log n}{\log r} +\frac{\log (1-\rho)}{\log \rho} +1$, we take, for $n$ sufficiently large,
\[
t^*=\left[\frac{4\log n}{\log r}\right].
\]
From the expression for $\gamma(n,t)$, one easily gets
\begin{equation*}
\left|\p\left(Q_{p,0}(u_n)\cap Q_{p,j}(u_n)\right)-\p\left(Q_{p,0}(u_n)\right)^2\right|\leq \frac{\rho^{j-1}}{1-\rho},
\end{equation*}
which implies that since $k=k_n\to \infty$ as $n\to\infty$ we have
\begin{align*}
n\sum_{j=t^*+1}^{[n/k]}\p(Q_{p,0}(u_n)\cap Q_{p,j}(u_n))&\leq n[n/k]\left(\p(Q_{p,0}(u_n))^2+\frac 1{n^3}\right)\\
&\leq \frac{n^2}{k} \p(X_0>u_n)^2+\frac{n^2}{k n^3}\xrightarrow[n\to\infty]{}0.
\end{align*}

Thus, we are left with the piece of history from time $0$ to $t^*$ to analyse. Recall that $u_n=1-\tau/n$ so that $n\p(X_0>u_n)\to \tau\geq 0$, as $n\to\infty$. Observe that $Q_{p,0}(u_n)$ occurs if and only if $X_0>u_n$ and $X_1\leq u_n$, which, for $n$ large enough, will only happen if $\epsilon_1< (r-1)/r$ which means that $X_1\leq (r-1)/r$. Besides, for $Q_{p,j}(u_n)$ to occur we must have $X_j>u_n$. Since $X_j=\epsilon_j+\rho \epsilon_{j-1}+\ldots+\rho^{j-1} X_1$, we have that, for very large $n$, there exists $\varsigma=\varsigma(n)$ such that $\epsilon_j=\epsilon_{j-1}=\ldots=\epsilon_{j-\varsigma}=(r-1)/r$, otherwise $X_j$ cannot exceed the level $u_n$.  Next, we compute a lower bound for $\varsigma$.
\begin{equation*}
\frac{r-1}r\left(1+\frac1r+\ldots+\frac1{r^\varsigma}\right)\geq 1-\frac\tau n\Rightarrow 1-\frac{1}{r^{\varsigma+1}}\geq 1-\frac{\tau} {n}\Rightarrow \varsigma+1\geq \frac{\log n}{\log r}-\frac{\log\tau}{\log r}.
\end{equation*}
Hence, we set
\[
\varsigma=\left[\frac{\log n}{\log r}-\frac{\log\tau}{\log r}\right]-1.
\]

For $j-\varsigma>1$ and using that $\epsilon_{j-\varsigma},\ldots,\epsilon_j$ are independent random variables which are also independent from $X_0$ and $X_1$, we have
\begin{align*}
\p\left(Q_{p,0}(u_n)\cap Q_{p,j}(u_n)\right) &\leq \p(Q_{p,0}(u_n)\cap \{X_j>u_n\})\\
&\leq \p\left(Q_{p,0}(u_n)\cap\{\epsilon_j=\epsilon_{j-1}=\ldots=\epsilon_{j-\varsigma}=(r-1)/r\}\right)\\
&\leq \p(Q_{p,0}(u_n))\p(\epsilon_{j-\varsigma}=(r-1)/r)\ldots\p(\epsilon_{j}=(r-1)/r)\\
&\leq \p(Q_{p,0}(u_n))\frac1{r^\varsigma}.
\end{align*}
If $j-\varsigma\leq 1$ then $\p(Q_{p,0}(u_n)\cap Q_{p,j}(u_n))=0$.

Observe that the occurrence of $Q_{p,j}(u_n)$ implies an exceedance of $u_n$ at time $j$ followed by the occurrence of the event $\epsilon_{j+1}<(r-1)/r$, which, in turn, implies that we have to wait at least a period of length $\varsigma$ before another exceedance of $u_n$ occurs: this means that there can only occur at most a $[t^*/\varsigma]+1$ number of $Q_{p,j}$ events with $j=1,\ldots, t^*$. Hence, we have
\[
n\sum_{j=1}^{t^*} \p(Q_{p,0}(u_n)\cap Q_{p,j}(u_n))\leq n([t^*/\varsigma]+1)\p(X_0>u_n)\frac1{r^\varsigma}.
\]
Finally, since there exists some constant $C>0$ such that $[t^*/\varsigma]+1\leq \frac{\frac{\log n}{\log r}-\frac{\log\tau}{\log r}}{\frac{4\log n}{\log r}-1}+1\leq C$,  for all $n\in\N$; $n\p(X_0>u_n)\to\tau$, as $n\to\infty$;  $(1/r)^\varsigma\leq (1/r)^{\frac{\log n}{\log r}-\frac{\log\tau}{\log r}}\to0$, as $n\to\infty$; we have
\[
\lim_{n\to\infty}n\sum_{j=1}^{t^*} \p(Q_{p,0}(u_n)\cap Q_{p,j}(u_n))=0.
\]

\bibliographystyle{/Users/amoreira/DropboxJorge/Dropbox/Bibliografia/novostyle}

\bibliography{ExtremalIndex}

%\bibliography{/Users/jmfreita/Dropbox/Bibliografia/ExtremalIndex}

\end{document}